\documentclass[11pt,a4paper,twoside,leqno]{article}

\usepackage{times}
\usepackage{amsmath,amssymb,amsthm,amsfonts}
\usepackage{a4wide}
\usepackage{exscale}   
\usepackage{enumerate}
\usepackage{mathrsfs}
\usepackage{latexsym}
\usepackage[latin1]{inputenc} 
\usepackage[T1]{fontenc}

\usepackage[all]{xy} 

\usepackage{stmaryrd}

\newcommand{\abs}[1]{\left\lvert#1\right\rvert}   
\newcommand{\norm}[1]{\left\lVert#1\right\rVert}   

\newcommand{\demph}[1]{{\it #1}}

\DeclareMathOperator{\kernel}{Kern}
\DeclareMathOperator{\id}{Id}
\DeclareMathOperator{\ev}{ev} 

\newcommand{\Cont}{{\mathscr C}} 

\DeclareMathOperator{\Lin}{L} 
\DeclareMathOperator{\Ad}{Ad}
\DeclareMathOperator{\Leb}{L} 

\DeclareMathOperator{\Komp}{K} 

\DeclareMathOperator{\KTh}{K}
\DeclareMathOperator{\KK}{KK}

\DeclareMathOperator{\Mult}{M}

\newcommand{\C}{\ensuremath{{\mathbb C}}}
\newcommand{\E}{\ensuremath{{\mathbb E}}}

\newcommand{\M}{\ensuremath{{\mathbb M}}}
\newcommand{\N}{\ensuremath{{\mathbb N}}}

\newcommand{\R}{\ensuremath{{\mathbb R}}}
\newcommand{\Z}{\ensuremath{{\mathbb Z}}}

\newcommand{\mA}{\ensuremath{{\mathcal A}}}
\newcommand{\mC}{\ensuremath{{\mathcal C}}}
\newcommand{\mD}{\ensuremath{{\mathcal D}}}
\newcommand{\mE}{\ensuremath{{\mathcal E}}}
\newcommand{\mF}{\ensuremath{{\mathcal F}}}
\newcommand{\mG}{\ensuremath{{\mathcal G}}}
\newcommand{\mM}{\ensuremath{{\mathcal M}}}
\newcommand{\mR}{\ensuremath{{\mathcal R}}}

\theoremstyle{plain}

\newtheorem {theorem} {Theorem}[section]
\newtheorem {lemma}[theorem] {Lemma}
\newtheorem {proposition} [theorem]{Proposition}
\newtheorem {corollary} [theorem]{Corollary}

\newtheorem* {theorem*} {Theorem}
\newtheorem* {proposition*} {Proposition}
\newtheorem* {lemma*}{Lemma}

\theoremstyle{definition}

\newtheorem {definition} [theorem]{Definition}

\newtheorem {deflemma} [theorem]{Definition and Lemma}

\newtheorem {remark} [theorem]{Remark}

\newcommand{\KKban}{{\mathrm {KK}}^{\ban}}
\newcommand{\KKbanG}{{\mathrm {KK}}^{\ban}_G}
\newcommand{\Eban}{\E^{\ban}}
\newcommand{\EbanG}{\E^{\ban}_G}

\DeclareMathOperator{\ban}{ban}
\DeclareMathOperator{\ssplit}{ss}
\DeclareMathOperator{\equi}{equiv}

\newcommand{\kk}{{\mathrm {kk}}}

\newcommand{\kkban}{\kk^{\ban}}
\newcommand{\kkbanG}{\kk^{\ban}_{G}}
\newcommand{\kkbancat}{\cat{kk}^{\ban}}

\newcommand{\kkbanGfunc}{\kk^{\ban}_G}
\newcommand{\kkbanGcat}{\cat{kk}^{\ban}_G}

\newcommand{\rkk}{{\mathrm {rkk}}}

\newcommand{\rkkban}{\rkk^{\ban}}

\newcommand{\Rkk}{{\mathrm {Rkk}}}

\newcommand{\Rkkban}{\Rkk^{\ban}}
\newcommand{\RkkbanG}{\Rkk^{\ban}_{G}}

\newcommand{\RkkbanGcat}{\cat{Rkk}^{\ban}_G}

\newcommand{\rRkk}{{\mathcal R}{\mathrm {kk}}}

\newcommand{\rRkkban}{\rRkk^{\ban}}
\newcommand{\rRkkbanG}{\rRkk^{\ban}_{G}}

\newcommand{\rRkkbanGcat}{{\bf \mathcal R}\cat{kk}^{\ban}_G}

\newcommand{\MoritabanGcat}{\cat{Mor}^{\ban}_G}

\newcommand{\LazyAnd}{\quad \text{and} \quad}

\newcommand{\rmd}{{\, \mathrm{d}}}
\newcommand{\zylinder}{{\mathrm Z}} 
\newcommand{\cone}{{\mathrm C}}  
\newcommand{\suspension}{{\mathrm \Sigma}}  

\DeclareMathOperator{\RKK}{\mR KK}

\newcommand{\RKKban}{\RKK^{\ban}} 
\newcommand{\RKKbanG}{\RKK^{\ban}_{G}} 

\newcommand{\gG}{\mG} 

\newcommand{\ContSect}{\Gamma}

\DeclareMathOperator{\MbanG}{\M^{\ban}_G} 

\DeclareMathOperator{\Moritaban}{Mor^{\ban}} 
\DeclareMathOperator{\MoritabanG}{Mor^{\ban}_G} 

\newcommand{\Moritabancat}{\cat{Mor}^{\ban}}

\DeclareMathOperator{\colim}{colim}

\newcommand{\cat}[1]{\text{\rm \bf  #1}}

\newcommand{\BanAlg}{\ensuremath{\cat{BanAlg}}}
\newcommand{\GBanAlg}{\ensuremath{G\text{-}\cat{BanAlg}}}

\DeclareMathOperator{\pt}{pt.}

\DeclareMathOperator{\SigmaHobanGcat}{\ensuremath{\mathbf{\Sigma}\cat{Ho}^{\ban}_G}}
\DeclareMathOperator{\SigmaHobanG}{\Sigma Ho^{\ban}_G}

\DeclareMathOperator{\can}{can}

\begin{document}

\title{$kk$-Theory for Banach Algebras II:\\ Equivariance and Green-Julg type theorems}
\author{Walther Paravicini}
\date{\today}
\maketitle

\begin{abstract}
\noindent We extend the definition of the bivariant $\KTh$-theory $\kk^{\ban}$ from plain Banach algebras to Banach algebras equipped with an action of a locally compact Hausdorff group $G$. We also define a natural transformation from Lafforgue's theory $\KKbanG$ into the new equivariant theory, overcoming some technical difficulties that are particular to the equivariant case. The categorical framework allows us to systematically define a descent homomorphism and to prove a Green-Julg theorem, a dual version of it and a generalised version that involves the action of a proper $G$-space. We also include a na{\"i}ve  Poncar\'{e} duality theorem.

\medskip 


\medskip

\noindent AMS 2010 {\it Mathematics subject classification:} Primary 19K35 ; Secondary 22D15, 46M18, 46H99

\end{abstract}

\noindent In \cite{Paravicini:14:kkban:arxiv}, we have defined an equivariant bivariant $\KTh$-theory $\kkban$ for Banach algebras that has a product and reasonable homological properties. This was achieved using triangulated categories and ensuring that the new theory has a universal property that implies that it is, in particular, Morita invariant and that there is a canonical natural transformation from Vincent Lafforgue's theory $\KKban$ into it. A key ingredient in the definition was the notion of a semi-split short-exact sequence of Banach algebras, i.e. a short exact sequence
$$
\xymatrix{  B \ \ar@{>->}[r]^{\iota}& D \ar@{->>}[r]^{\pi}& A}
$$
of Banach algebras where $\iota$ and $\pi$ are continuous homomorphisms, $\pi$ is surjective, $\iota$ is injective and $\iota(B) = \kernel(\pi)$ and such that the map $\pi$ has a continuous linear split $\sigma$. 

Now let $G$ be a locally compact Hausdorff group. In order to define an equivariant theory $\kkbanG$ one has to decide on a definition of $G$-equivariant semi-split extensions: We consider $G$-actions on the Banach algebras $A$, $D$ and $B$ and assume that $\iota$ and $\pi$ are $G$-equivariant. The extra complexity that led to the separation of the non-equivariant from the equivariant case in  \cite{Paravicini:14:kkban:arxiv} in order to keep the exposition clearer is related to the question whether or not $\sigma\colon A \to D$ should be assumed to be $G$-equivariant as well.

In the case of an odd cycle, a na\"{i}ve extension of the definition of the natural transformation from $\KKban$ to $\kkban$ given in \cite{Paravicini:14:kkban:arxiv} to the equivariant case leads to an extension with a (possibly) non-equivariant split. Interestingly, one can always replace such an extension with a Morita equivalent extension with an equivariant split (compare \cite{Thomsen:01} for a C$^*$-algebraic version of this). But in the case of an even cycle $(E,T)\in \EbanG(A,B)$, where $A$ and $B$ are $G$-Banach algebras, the machinery of double split extensions that is used in \cite{Paravicini:14:kkban:arxiv} to define the morphism in $\kkban$ attached to $(E,T)$ seems not to work in the equivariant case if $T$ itself is not equivariant. 

So the question arises if we are allowed to replace the operator $T$, that is only equivariant up to compact operators, with an operator that is actually equivariant. We show that this is always possible but we have to replace the algebras $A$ and $B$ with $G$ equivariantly Morita equivalent algebras $A'$ and $B'$, respectively, to obtain enough space; as $\kkbanG$ is defined in a way that ensures Morita invariance, this will be good enough for our purposes. Compare \cite{Thomsen:01} and \cite{Meyer:00:EquivariantKK} for C$^*$-algebraic considerations of a similar kind. The Morita equivalences that we use for this construction are the same that have to be used for the result on short exact sequences mentioned above, so we deal with the construction for extensions in Section~\ref{Subsection:EquivariantExtensions} before we come to the construction for cycles in Section~\ref{Subsection:EquivariantCycles}.

In Section~\ref{Section:DefinitionAndBasicProperties}, we generalize the definition of $\kkban$ from \cite{Paravicini:14:kkban:arxiv} to the equivariant case and construct the natural transformation from $\KKbanG$ into the new theory $\kkbanG$. Because of the problems mentioned above, the construction of this comparison transform is somewhat delicate and extensive although the general idea is the same as in the non-equivariant case. It is compatible with the descent homomorphism for unconditional completions that we define. 

The  theorems on $\kkbanG$ that are proved in this article are the following:

\begin{theorem}[The Green-Julg theorem for $\kkban$]\label{Theorem:GreenJulg:Einleitung}
Let $G$ be a compact Hausdorff group. Then
$$
\kkbanG(A, B) \cong \kkban(A, \Leb^1(G,B)),
$$
naturally, where $A$ is a Banach algebra (equipped with trivial $G$-action) and $B$ is a $G$-Banach algebra.
\end{theorem}

Note that this becomes the ordinary Green-Julg theorem (for $\Leb^1$) if $A=B=\C$. The version stated here says that the functor that sends a Banach algebra $A$ to itself, thought of as a $G$-Banach algebra with trivial $G$-action, and the functor that sends a $G$-Banach algebra $B$ to $\Leb^1(G,B)$, are adjoint.\footnote{After obtaining this result in the Banach algebraic context I discovered that the same approach (also for the dual version of the theorem) has already been presented in \cite{Meyer:08} in a C$^*$-algebraic setting and, more recently, in \cite{Ellis:14} in an algebraic situation. The proofs in the Banach algebraic situation, however, seem to be somewhat more involved.} Indeed, this is how it is proved: We check the unit-co-unit equations that constitute the adjunction. 

Dually, we have another adjunction that we prove similarly:

\begin{theorem}[The dual Green-Julg Theorem for $\kkban$]\label{Theorem:DualGreenJulg:Einleitung}
Let $G$ be a discrete group. Then
$$
\kkbanG(A, B) \cong \kkban(\ell^1(G, A), B),
$$
naturally, where $A$ is a $G$-Banach algebra and $B$ is a Banach algebra (equipped with the trivial action).
\end{theorem}

Finally, we define a generalisation $\rRkkbanG(X; \cdot, \cdot)$ of $\kkbanG$ that is suited for $G$-Banach algebras that are fibred over the locally compact $G$-space $X$. The Green-Julg theorem mentioned above generalises to the following theorem:

\begin{theorem}[The Green-Julg Theorem for proper group actions]\label{Theorem:GreenJulg:generalised:Einleitung}
Let $G$ be a locally compact Hausdorff group acting properly and continuously on a locally compact Hausdorff space $X$. Then
$$
\rRkkbanG(X; A_{\tau}, B) \cong \rRkkban(X/G; A, \Leb^1(G,B)),
$$
naturally, where $A$ is a $\Cont_0(X/G)$-Banach algebra, $B$ is a $G$-$\Cont_0(X/G)$-Banach algebra and $A_{\tau}:=\Cont_0(X) \otimes_{\Cont_0(X/G)} A$, equipped with the $G$-action coming from the canonical action of $G$ on $\Cont_0(X)$.
\end{theorem}

This theorem is the $\kkban$-version of the main theorem of \cite{Paravicini:10:GreenJulg:erschienen} that involved $\KKban$, and the proof is considerably easier.  As a consequence, we show 

\begin{corollary}\label{Corollary:GGJ:Einleitung} If, in the above theorem, $X/G$ is compact and $A=\Cont_0(X/G)$ then we have a natural isomorphism
$$
\rRkkbanG(X; \Cont_0(X), B)\ \cong \ \KTh_0(\Leb^1(G,B)).
$$
\end{corollary}

\section{Preliminaries}

Let $G$ be a locally compact Hausdorff group.  In what follows, we use the notation of \cite{Paravicini:14:kkban:arxiv}. Before we discuss the methods, alluded to in the introduction, of making splits or operators equivariant, we first convince ourselves that the theory of Morita morphisms and Morita equvalences between (possibly degenerate) Banach algebras carries over from \cite{Paravicini:14:kkban:arxiv} without any surprises.

\subsection{The equivariant Morita category}

Recall from \cite{Paravicini:07:Morita:richtigerschienen} that there is an ``$G$-equivariant Morita category'', the objects of which are the non-degenerate $G$-Banach algebras and the morphism set between an object $A$ and an object $B$ is the set $\MoritabanG(A,B)$ of homotopy classes of $G$-equivariant Morita cycles between $A$ and $B$. You can think of the Morita cycle as the composition of a (usual) homomorphism of Banach algebras and a Morita equivalence. A homotopy from $A$ to $B$ is a Morita cycle between $A$ and $B[0,1]$. 

In \cite{Paravicini:14:kkban:arxiv}, the definition of the (non-equivariant) Morita category is extended to possibly degenerate Banach algebras. The further extension to possibly degenerate $G$-Banach algebras can be achieved by adding $G$-actions at the appropriate places; for example, the definition of an equivariant Morita cycle reads as follows:

 \begin{definition}[Morita cycle]
A \demph{$G$-equivariant Morita cycle} $(F,\varphi)$ from $A$ to $B$ is a $G$-Banach $B$-pair $F$ together with a $G$-equivariant homomorphism $\varphi\colon A \to \Lin_B(F)$ such that there is a $k\in \N$ satisfying $\varphi(A)^k \subseteq \Komp_B(F)$. The class of all Morita cycles from $A$ to $B$ will be denoted by $\MbanG(A,B)$.
\end{definition}

If $(F, \varphi)\in \MbanG(A,B)$ then we will sometimes simply write $F$ for this cycle and suppress the left action $\varphi$ in the notation.

Taking this as a starting point and generalising the definitions from \cite{Paravicini:14:kkban:arxiv}, one obtains a $G$-equivariant version $\MoritabanGcat$ of the Morita category $\Moritabancat$ together with a canonical functor $\varphi \mapsto \MoritabanG(\varphi)$ from the category $\GBanAlg$ of $G$-Banach algebras and equivariant homomorphisms to $\MoritabanGcat$.

In particular, there is a notion of a $G$-equivariant Morita equivalence between (possibly degenerate) $G$-Banach algebras; such a $G$-Morita equivalence $E$ from $A$ to $B$ gives an isomorphism $\MoritabanG(E)$ from $A$ to $B$ in the $G$-equivariant Morita category.

Also the definition of concurrent homomorphisms between Morita equivalences generalises to the equivariant case. We obtain the following condition on when two Morita cycles give the same Morita morphism:

\begin{lemma}\label{Lemma:ConcurrentEquivalences:Equivariant}
Let ${_{\chi}} \Phi_\psi$  be a $G$-equivariant concurrent homomorphism from a $G$-equivariant Morita equivalence ${_A}E_B$ to a $G$-equivariant Morita equivalence ${_{A'}}E'_{B'}$ with ($G$-equivariant) coefficient maps $\varphi\colon A\to A'$ and $\psi\colon B \to B'$. Then the mapping cone of $\Phi$ induces a homotopy that yields the equation
$$
\MoritabanG(\psi) \circ \MoritabanG(E) = \MoritabanG(E') \circ \MoritabanG(\chi) \quad \in \quad \MoritabanG(A, B') .
$$
\end{lemma}

We summarise the situation of the preceding lemma in the diagram
$$
\xymatrix{
 A \ar[rr]^-{E} \ar[d]_{\varphi}& \ar@{=>}[d]|{\Phi}&B \ar[d]^{\psi}  \\ 
 A'\ar[rr]_{E'}&&B' 
}
$$
that commutes in $\MoritabanGcat$.

By construction, the canonical functor from $\GBanAlg$ to $\MoritabanGcat$ is invariant under $G$-equivariant homotopies and under $G$-equivariant Morita equivalences. As in \cite{Paravicini:14:kkban:arxiv}, Theorem 1.17, one can show that this functor is universal with this property.

\begin{theorem}\label{Theorem:LiftToMoritabanG}
The functor $\MoritabanG \colon \GBanAlg \to \MoritabanGcat$ is the universal $G$-homotopy invariant and $G$-Morita invariant functor on $\GBanAlg$: 
Let $\mC$ be a category and $\mF \colon \GBanAlg \to \mC$ a functor. Then $\mF$ is $G$-homotopy invariant and $G$-Morita invariant if and only if $\mF$ factors through  $\MoritabanG \colon \GBanAlg \to \MoritabanGcat$. The factorisation, if it exists, is unique. 
\end{theorem}

Now let $\mA(G)$ be an unconditional completion of $\Cont_c(G)$. Consider the descent functor $\mA(G, \cdot)$ from $\GBanAlg$ to $\BanAlg$. 

\begin{proposition}\label{Proposition:Descent:Morita}
There is a unique functor, that we also denote $\mA(G, \cdot)$, from $\MoritabanGcat$ to $\Moritabancat$ that makes the following diagram commutative
$$
\xymatrix{\GBanAlg \ar[rr]^-{\MoritabanG} \ar[d]^{\mA(G, \cdot)}&& \MoritabanGcat \ar@{-->}[d]^{\mA(G, \cdot)}\\
\BanAlg \ar[rr]^-{\Moritaban} && \Moritabancat
}
$$
\end{proposition}
\begin{proof}
Let $\mF\colon \GBanAlg \to \Moritabancat$ be the functor $\Moritaban \circ \mA(G,\cdot)$. We show that it is $G$-ho\-mo\-to\-py invariant and $G$-Morita invariant so that the claim follows from Theorem~\ref{Theorem:LiftToMoritabanG}. 
\begin{itemize}
\item The functor $\mA(G, \cdot)$ sends $G$-homotopic morphisms to homotopic morphisms: Let $\varphi \colon A \to B[0,1]$ be a $G$-equivariant homotopy between $\varphi_0$ and $\varphi_1$. Then $\mA(G, \varphi)$ is a morphism from $\mA(G, A)$ to $\mA(G,B[0,1])$. If you compose it with the canonical contractive morphism from $\mA(G, B[0,1])$ to $\mA(G, B)[0,1]$ then you obtain a homotopy between $\mA(G,\varphi_0)$ and $\mA(G, \varphi_1)$. So $\mF$ is $G$-homotopy invariant.

\item The functor $\mA(G, \cdot)\colon \GBanAlg \to \BanAlg$ sends $G$-Morita equivalences to Morita equivalences: Let $E$ be a $G$-equi\-va\-riant Morita equivalence between $G$-Banach algebras $A$ and $B$. Then $\mA(G,E)$ is a Morita equivalence between $\mA(G,A)$ and $\mA(G,B)$. If $\iota$ denotes the canonical inclusion of $A$ into the linking algebra $L$ of $E$, then $\mA(G, \iota)$ is the inclusion of $\mA(G, A)$ in $\mA(G, L)$. Because $\Moritaban\colon \BanAlg \to \Moritabancat$ is Morita invariant, it follows that $\mF$ is $G$-Morita invariant.
\end{itemize}\end{proof}

\subsection{Equivariant extensions and equivariant splits}\label{Subsection:EquivariantExtensions}

In this section we discuss how every extension of $G$-Banach algebras with continuous linear split can, up to Morita equivalence, be replaced with an extension with equivariant split.

\begin{definition} We define three classes $\mE^G\supseteq \mE^G_{\ssplit} \supseteq \mE^G_{G-\ssplit}$ of extensions of $G$-Banach algebras:

\begin{enumerate}
\item The class $\mE^G$ is defined to be the class of all \demph{($G$-equivariant) extensions} of $G$-Banach algebras, i.e., extensions of $G$-Banach algebras of the form 
$$
\xymatrix{  B \ \ar@{>->}[r]^{\iota}& D \ar@{->>}[r]^{\pi}& A}
$$
where $\iota$ and $\pi$ are $G$-equivariant continuous homomorphisms, $\pi$ is surjective, $\iota$ is injective (with closed range) and $\iota(B) = \kernel(\pi)$.

\item The class $\mE_{\ssplit}^G$ is defined to be the class of all \demph{semi-split ($G$-equivariant) extensions} of $G$-Banach algebras, i.e., extensions of the above form where $\iota$ and $\pi$ are $G$-equivariant continuous homomorphisms and $\pi$ has a continuous linear (not necessarily $G$-equivariant) split.

\item The class $\mE_{G-\ssplit}^G$ is defined to be the sub-class of $\mE^G_{\ssplit}$ of \demph{equivariantly semi-split ($G$-equivariant) extensions} of $G$-Banach algebras, i.e., $G$-equivariant extensions that permit a  \emph{$G$-equivariant} linear continuous split.
\end{enumerate}
\end{definition}

Extensions arising from odd $G$-equivariant Kasparov cycles are $G$-equivariant extensions that permit a continuous linear split but that do not come naturally with a $G$-equivariant split (see the discussion in Subsection~\ref{Subsection:Comparison}). However, it is always possible to replace a given extension in $\mE^G_{\ssplit}$ another extension in $\mE^G_{G-\ssplit}$ that is ``$G$-equivariantly Morita equivalent''. This fact by itself, interesting as it is, does not seem to help directly in the case of even Kasparov modules, but the tools that we develop can and will be used in the even case as well.

In a first step, we are now going to define, for every $G$-Banach algebra $A$, another $G$-Banach algebra called $\Leb^1(G\ltimes G, A)$ that is equivariantly Morita equivalent to $A$. For further reference, we even give the necessary definitions for Banach spaces etc..

Note that the following definition is a variant (that can even be regarded as a special case) of the corresponding definitions in Section~5 of \cite{Paravicini:07:Induction:erschienen}. Compare also the definitions in Section~1.3 of \cite{Lafforgue:06}.
  
\begin{definition}
\begin{enumerate}
\item Let $E$ be a Banach space. Define $\Leb^1(G\ltimes G, E)$ to be the completion of $\Cont_c(G\times G, E)$ with respect to the norm:
$$
\norm{\xi}_1:= \sup_{t\in G}\int_G \norm{\xi(s,t)}_E \rmd s.
$$
\item If $E_1$, $E_2$ and $F$ are Banach spaces, with $E_2$ a $G$-Banach space, i.e., $E$ carries a strongly continuous isometric action of $G$, and $\mu\colon E_1 \times E_2 \to F$ is a continuous bilinear map, then we define the convolution
$$
(\xi_1 * \xi_2)(s,t) := \int_G \mu\left(\xi_1(r, t),\ r \xi_2(r^{-1}s, r^{-1}t)\right) \rmd r, \qquad s,t\in G.
$$
for all $\xi_1 \in \Cont_c(G\times G, E_1)$ and $\xi_1 \in \Cont_c(G\times G, E_1)$. This product extends to a continuous bilinear map $* \colon \Leb^1(G\ltimes G, E_1) \times \Leb^1(G\ltimes G, E_2) \to \Leb^1(G\ltimes G, F)$ such that $\norm{*} \leq \norm{\mu}$.
\item If $E$ is a Banach space, then $\Leb^1(G \ltimes G, E)$ carries an isometric strongly continuous action of $G$ which is given on $\Cont_c(G\times G, E)$ by
$$
(r\xi)(s,t) := \xi(s,tr), \qquad r,s,t\in G, \xi \in \Cont_c(G\times G, E).
$$
\end{enumerate}
\end{definition}

\begin{definition}
Let $E$ and $F$ be Banach spaces and $T\in \Lin(E,F)$ a continuous linear map from $E$ to $F$. On $\Cont_c(G \times G, E)$, define the map $\Leb^1(G\ltimes G, T)$ by 
$$
\Leb^1(G\ltimes G, T)(\xi) (s,t):= T\left(\xi(s,t)\right),
$$
for all $s,t\in G$ and $\xi\in \Cont_c(G\times G, E)$. Then $\Leb^1(G\ltimes G, T)$ extends to a continuous linear and \emph{$G$-equivariant} map from $\Leb^1(G\ltimes G, E)$ to $\Leb^1(G\ltimes G, F)$ such that $\norm{\Leb^1(G\ltimes G, T)} \leq \norm{T}$.  This definition is functorial. 
\end{definition}

\begin{lemma}
Let $A$ be a $G$-Banach algebra. Then $\Leb^1(G\ltimes G, A)$ is a $G$-Banach algebra when equipped with the convolution product and the above $G$-action. 
\end{lemma}

\begin{remark}
The algebra $\Leb^1(G\ltimes G, A)$ can be thought of as the $\Leb^1$ version of the groupoid crossed product of $A$ by the transport groupoid $G \ltimes G$. The $G$-action on $\Leb^1(G\ltimes G, A)$ is given by the right action of $G$ on $G$ that commutes with the left $G$-action on $G$.
\end{remark}

\begin{proposition}\label{Proposition:MoritaEquivalenceLeb1}
Let $A$ be a $G$-Banach algebra. Then $A$ and $\Leb^1(G\ltimes G, A)$ are $G$-equivariantly Morita equivalent in the sense of \cite{Paravicini:07:Morita:richtigerschienen} and \cite{Lafforgue:04}; the Morita equivalence is given by the pair $(\Leb^1(G,A),\Cont_0(G,A))$, equipped with the following operations:
\begin{eqnarray*}
(\alpha \cdot \xi^>)(t)&:=& \int_{G} \alpha(r,t)\ r \xi^>(r^{-1}t) \rmd r\\
(\xi^> a)(t) &:=& \xi^>(t)\ t a\\
(u \xi^>)(t) &:=& \xi^>(tu)\\
(\xi^< \cdot \alpha)(s) &:=& \int_{G} \xi^<(r)\ r \alpha(r^{-1}s, r^{-1}) \rmd r\\
(u \xi^<)(s) &:=& u \xi^<(u^{-1}s)\\
(a \xi^<)(s) &:= & a \xi^<(s)\\
\langle \xi^<, \xi^>\rangle_A &:=& \int_{G} \xi^<(r)\  r \xi^>(r^{-1}) \rmd r\\
{_{\Leb^1}} \langle \xi^>, \xi^<\rangle (s,t) &:=& \xi^>(t) \ t\xi^<(t^{-1}s),
\end{eqnarray*}
for all $s,t,u \in G$, $a\in A$, $\xi^< \in \Cont_c(G,A) \subseteq \Leb^1(G,A)$, $\xi^>\in \Cont_c(G,A) \subseteq \Cont_0(G,A)$ and $\alpha \in \Cont_c(G\times G, A)$.
\end{proposition}
\begin{proof}
By direct calculation one can see that $\Leb^1(G,A)$ is a $G$-equivariant Banach $A$-$\Leb^1(G\ltimes G, A)$-bimodule with the above operations and, similarly, that $\Cont_0(G,A)$ is a $G$-equivariant Banach $\Leb^1(G\ltimes G, A)$-$A$-bimodule. The $A$-valued inner product satisfies 
$$
\norm{\langle \xi^<, \xi^>\rangle_A} \leq \norm{\xi^<}_1 \norm{\xi^>}_{\infty}
$$
and turns $(\Leb^1(G,A),\Cont_0(G,A))$ into a $G$-equivariant Banach $A$-pair as can be checked by direct calculation. Similarly, one checks that $(\Cont_0(G,A),\Leb^1(G,A))$ is a $G$-equivariant Banach $\Leb^1(G\ltimes G, A)$-pair, which means in particular that we have
$$
\norm{{_{\Leb^1}}\langle \xi^>, \xi^<\rangle} \leq \norm{\xi^>}_{\infty}\norm{\xi^<}_1.
$$
 The actions and the inner product are compatible. Moreover, it is easy to check that the closed linear span of $\langle \Leb^1(G, A), \Cont_0(G,A)\rangle_A$ is the closed linear span $AA$ of $\{a\cdot a': \ a,a'\in A\} \subseteq A$. Also, the linear span of ${_{\Leb^1}} \langle \Cont_c(G,A), \Cont_c(G,A)\rangle$ is dense in the inductive limit topology in $\Cont_c(G \times G, AA)$, and it follows that  the span of ${_{\Leb^1}} \langle \Cont_0(G,A), \Leb^1(G,A)\rangle$ is dense in $\Leb^1(G\ltimes G, A^2)$ which can be shown to be $\Leb^1(G\ltimes G, A)^2$.
So it follows that  $(\Leb^1(G,A),\Cont_0(G,A))$ is a Morita equivalence between $\Leb^1(G\ltimes G, A)$ and $A$ in the sense of \cite{Paravicini:07:Morita:richtigerschienen}.\end{proof}

\begin{theorem}\label{Theorem:Extension:GequivariantSplit}
Let $\xymatrix{  B \ \ar@{>->}[r]^{\iota}& D \ar@{->>}[r]^{\pi}& A}$ be a semi-split $G$-equivariant extension with continuous linear split $\sigma$. Then the $G$-equivairant extension
$$
\xymatrix{  \Leb^1(G\ltimes G, B) \ \ar@{>->}[rr]^-{\Leb^1(G\ltimes G, \iota)}&& \Leb^1(G\ltimes G, D) \ar@{->>}[rr]^-{\Leb^1(G\ltimes G,\pi)}& &\Leb^1(G\ltimes G, A)}
$$
is equivariantly semi-split with equivariant continuous linear split $\Leb^1(G\ltimes G, \sigma)$.
\end{theorem}

\subsection{Equivariant cycles and equivariant operators}\label{Subsection:EquivariantCycles}

Let $A$ and $B$ be $G$-Banach algebras. The class $\EbanG(A,B)$ is defined as in \cite{Lafforgue:02} or rather as in the proof of Theorem~1.27 of \cite{Paravicini:14:kkban:arxiv} if you do not want to restrict yourself to non-degenerate Banach algebras. Similarly, define the notion of homotopy on $\EbanG(A,B)$ as in \cite{Lafforgue:02} or \cite{Paravicini:14:kkban:arxiv} and call the resulting group of equivalence classes $\KKbanG(A,B)$. 

In this section, we discuss how one can replace any element $(E,T)$ of $\EbanG(A,B)$ with a version for which the operator $T$ is $G$-equivariant. This can be done in two natural ways: The perhaps the most natural replaces both algebras $A$ and $B$, with Morita equivalent version called $A'=\Leb^1(G\ltimes G, A)$ and $B'=\Leb^1(G\ltimes G, B)$. The cycle corresponding to $(E,T)$ in $\EbanG(A',B')$ is obtained in a very systematic way. A little less systematic, but smaller in a certain sense is another version of the same trick which finds the equivariant cycle corresponding to $(E,T)\in \EbanG(A,B)$ in $\EbanG(A',B)$. It is this second version that we are going to use later on. Both versions amount to the same cycle if one uses the Morita equivalence between $B'$ and $B$ to identify $\KKbanG(A',B')$ and $\KKbanG(A',B)$. 

\begin{definition}
Let $\E^{\ban}_{G, \equi}(A,B)$ be the subclass of all \demph{$G$-equivariant cycles} in $\EbanG(A,B)$, i.e., of all cycles $(E,T)$  such that $T$ is $G$-equivariant. Two such cycles are $G$-equivariantly homotopic if there exists a $G$-equivariant homotopy between them. Define $\KK^{\ban}_{G, \equi}(A,B)$ to be the class of all $G$-equivariant homotopy classes in $\E^{\ban}_{G, \equi}(A,B)$. It is an abelian group. 
\end{definition}

\subsubsection{Version 1: $\KKbanG(A,B) \to \KK^{\ban}_{G,\equi}(\Leb^1(G\ltimes G, A), \Leb^1(G\ltimes G, B))$}

The following definition is a refinement of the corresponding definition in Section~5 of \cite{Paravicini:07:Induction:erschienen}, compare also \cite{Lafforgue:06} and \cite{Paravicini:07}.

\begin{definition}[The transformation for Banach pairs]
Let $E$ be a $G$-Banach $B$-pair. Define $\Leb^1(G\ltimes G, E)$ to be the pair $(\Leb^1(G\ltimes G, E^<), \Leb^1(G\ltimes G, E^>))$ equipped with operations given by convolution, i.e., with
\begin{eqnarray*}
(\xi^> * \beta)(s,t) &:=& \int_{G}  \xi^>(r,t) \ r \beta(r^{-1}s, r^{-1}t) \rmd r\\
(\beta * \xi^<)(s,t) &:=& \int_{G} \beta(r,t) \ r \xi^<(r^{-1}s, r^{-1}t) \rmd r\\
\langle \xi^<, \xi^>\rangle (s,t) &:=& \int_{G} \langle \xi^<(r,t), \ r\xi^>(r^{-1}s, r^{-1}t) \rangle \rmd r
\end{eqnarray*}
for all $s,t\in G$, $\xi^< \in \Cont_c(G \times G, E^<)$, $\xi^> \in \Cont_c(G \times G, E^>)$ and $\beta \in \Cont_c(G \times G, B)$. Equipped with the $G$-actions defined above, this defines a $G$-Banach $\Leb^1(G\ltimes G, B)$-pair. If $E$ carries a compatible left action of some Banach algebra $A$, then similar convolution formulas define a compatible left action of $\Leb^1(G\ltimes G, A)$ on $\Leb^1(G \ltimes G, E)$.

\end{definition}

The statements in the preceding definition as well as in the following can be checked by direct calculation.

\begin{definition}[The transformation for operators]
Let $E$ and $F$ be $G$-Banach $B$-pairs and $T=(T^<,T^>)\in \Lin_B(E,F)$. Define 
$$
\Leb^1(G\ltimes G, T):=(\Leb^1(G\ltimes G, T)^<, \Leb^1(G\ltimes G, T)^>) \in \Lin_{\Leb^1(G\ltimes G, B)}\left(\Leb^1(G\ltimes G, E), \Leb^1(G\ltimes G, F)\right)
$$
where $\Leb^1(G\ltimes G, T)^>$ is defined as above, i.e.,
$$
\Leb^1(G\ltimes G, T)^>(\xi^>) (s,t) = T^>\left( \xi^>(s,t)\right),
$$
for all $s,t\in G$ and $\xi^> \in \Cont_c(G\times G, E^>)$, whereas $\Leb^1(G\ltimes G, T)^<$ is defined, somewhat differently, by
$$
\Leb^1(G\ltimes G, T)^<(\eta^<) (s,t) = s \ T^<\left(s^{-1} \eta^<(s,t)\right),
$$
for all $s,t\in G$ and $\eta^< \in \Cont_c(G\times G, F^<)$. We have $\norm{\Leb^1(G\ltimes G, T)} \leq \norm{T}$.

This construction is functorial and $\Leb^1(G\ltimes G, T)$ is equivariant.
\end{definition}

\begin{lemma}[The transformation for $\KKbanG$-cycles]\label{Lemma:L1Transformation:ForCycles}
Let $(E,T)$ be in $\EbanG(A,B)$. Then
$$
\Leb^1(G\ltimes G, (E,T)):= \left(\Leb^1(G\ltimes G, E), \Leb^1(G\ltimes G, T)\right) \ \in \ \E^{\ban}_{G,\equi}(\Leb^1(G\ltimes G, A), \Leb^1(G\ltimes G, B)).
$$ 
\end{lemma}
\begin{proof}
See \cite{Paravicini:07}, Paragraph~5.2.8, compare \cite{Lafforgue:06}, Lemme~1.3.5 or \cite{Paravicini:07:Induction:erschienen}, Section 5.4.
\end{proof}


\begin{proposition}
Let $(E_0,T_0)$ and $(E_1, T_1)$ be homotopic elements of $\EbanG(A,B)$. Then there is a $G$-equivariant homotopy from $\Leb^1(G\ltimes G, (E_0, T_0))$ to $\Leb^1(G\ltimes G, (E_1, T_1))$ in $\E^{\ban}_{G,\equi}(\Leb^1(G\ltimes G, A), \Leb^1(G\ltimes G, B))$. We obtain a natural transformation
$$
\KKbanG(A,B) \ \to\  \KK^{\ban}_{G,\equi}(\Leb^1(G\ltimes G, A), \Leb^1(G\ltimes G, B)).
$$
\end{proposition}
\begin{proof}
See \cite{Paravicini:07}, Proposition~5.2.23.
\end{proof}

\subsubsection{Version 2: $\KKbanG(A,B) \to \KK^{\ban}_{G,\equi}(\Leb^1(G\ltimes G, A), B)$}

\begin{definition}\label{Definition:ContNullVomPaar}
Let $E$ be a $G$-Banach $A$-$B$-pair. Define $\Cont_0(G, E)$ to be the pair  $(\Leb^1(G,E^<),\Cont_0(G,E^>))$ equipped with operations given by convolution, i.e., with
\begin{eqnarray*}
(\alpha \cdot \xi^>)(t)&:=& \int_{G} \alpha(r,t) \ r\xi^>(r^{-1}t) \rmd r\\
(\xi^> b)(t) &:=& \xi^>(t) \ tb\\
(u \xi^>)(t) &:=& \xi^>(tu)\\
(\xi^< \cdot \alpha)(s) &:=& \int_{G} \xi^<(r) \ r\alpha(r^{-1}s, r^{-1}) \rmd r\\
(u \xi^<)(s) &:=& u \xi^<(u^{-1}s)\\
(b \xi^<)(s) &:= & b \xi^<(s)\\
\langle \xi^<, \xi^>\rangle_B &:=& \int_{G} \langle \xi^<(r), \ r \xi^>(r^{-1}) \rangle \rmd r\\
\end{eqnarray*}
for all $s,t,u \in G$, $b\in B$, $\xi^< \in \Cont_c(G,E^<) \subseteq \Leb^1(G,E^<)$, $\xi^>\in \Cont_c(G,E^>) \subseteq \Cont_0(G,E^>)$ and all $\alpha \in \Cont_c(G\times G, A)$. Equipped with the $G$-actions defined above, this defines a $G$-Banach $\Leb^1(G\ltimes G, A)$-$B$-pair.

\end{definition}

As above, we define the transformation also for operators:

\begin{definition} \label{Definition:ContNullVomZykel}
Let $E$ and $F$ be $G$-Banach $B$-pairs and $T=(T^<,T^>)\in \Lin_B(E,F)$. Define 
$$
\Cont_0(G,T):=(\Cont_0(G,T)^<, \Cont_0(G,T)^>) \in \Lin_{B}\left(\Cont_0(G, E), \Cont_0(G, F)\right)
$$
where $\Cont_0(G,T)^>$ is defined as
$$
\Cont_0(G,T)^>(\xi^>) (t) = T^>\left(\xi^>(t)\right),
$$
for all $t\in G$ and $\xi^> \in \Cont_c(G, E^>)$, and $\Cont_0(G,T)^<$ is defined by
$$
\Cont_0(G,T)^<(\eta^<) (t) = t^{-1} \ T^<\left(t \eta^<(t)\right),
$$
for all $t\in G$ and $\eta^< \in \Cont_c(G, F^<)$. We have $\norm{\Cont_0(G,T)} \leq \norm{T}$. This construction is functorial and $\Cont_0(G,T)$ is equivariant.
\end{definition}

\begin{lemma}
If $(E,T)$ is in $\EbanG(A,B)$, then
$$
\Cont_0(G,(E,T)):=(\Cont_0(G, E),\Cont_0(G, T)) \ \in \ \E^{\ban}_{G, \equi}(\Leb^1(G\ltimes G, A), B).
$$
It is compatible with homotopies of (equivariant) cycles.
\end{lemma}
\begin{proof}
The arguments are basically the same as in the proof of Lemma~\ref{Lemma:L1Transformation:ForCycles}. What you have to show is that the convolution with elements of $\Leb^1(G\ltimes G, \Komp_B(E))$ define compact operators on $\Cont_0(G, E)$. This is done analogously to \cite{Paravicini:07}, Paragraph~5.2.7.
\end{proof}



Because we will use the resulting map on the level of $\KKban$-elements later on, we will give it a name:

\begin{definition}\label{Definition:GammaAB}
Let 
$$
\gamma_{A,B} \colon \KKbanG(A,B) \to \KK^{\ban}_{G, \equi}(\Leb^1(G \ltimes G, A),B)
$$
be the map $(E,T) \mapsto \Cont_0(G, (E,T))$ on the level of homotopy classes.
\end{definition}

It is a natural homomorphism.

\subsubsection{Comparison of the two versions}

\begin{proposition}
Let $(E,T)$ be an element of $\EbanG(A,B)$. Then there is a canonical equivariant concurrent homomorphism 
$$
\left(\Leb^1(G\ltimes G, E) \otimes_{\Leb^1(G\ltimes G, B)} (\Leb^1(G,B), \Cont_0(G,B)),\ \Leb^1(G\ltimes G,T)\otimes 1\right)  \quad \to \quad \Cont_0( G, (E,T))
$$
which induces a $G$-equivariant homotopy of these two cycles in $\E^{\ban}_{G,\equi}(\Leb^1(G\ltimes G, A), B)$.
\end{proposition}
\begin{proof}[Sketch of proof]
It is straightforward to see that the canonical concurrent homomorphism from $\Leb^1(G\ltimes G, E) \otimes_{\Leb^1(G\ltimes G, B)} (\Leb^1(G,B), \Cont_0(G,B))$ to $(\Leb^1(G,E^<), \Cont_0(G,E^>))$ intertwines $\Leb^1(G\ltimes G,T)\otimes 1$ and $\Cont_0(G, T)$. It is also easy to check that this homomorphism is equivariant. To see that this homomorphism induces a homotopy one uses the criterion given in Theorem~3.1 of \cite{Paravicini:07:Morita:richtigerschienen}; the technical condition is checked by a careful inspection of how elements of $\Leb^1(G\ltimes G, \Komp(E))$ act on the two involved Banach pairs. The fact that the concurrent homomorphism is equivariant leads to the equivariance of the resulting homotopy.
\end{proof}

\section{Equivariant $\kkban$}\label{Section:DefinitionAndBasicProperties}

In this section, we define $\kkbanG$ and collect its basic properties that carry over from the non-equivariant case without much ado. The comparison map from $\KKbanG$ to $\kkbanG$ will keep us a little busy because we refrain from identifying Morita equivalent Banach algebras for the time being. We construct a descent transformation for $\kkbanG$ to $\kkban$ and, in Subsection~\ref{Subsection:CompatibilityComparisonDescent}, we show that the descent in $\KKban$ that was described by Lafforgue in \cite{Lafforgue:02} is compatible with the descent map in $\kkbanG$ via the comparison maps.

\subsection{The definition}

For every locally compact space $X$ and every $G$-Banach algebra $A$ define $AX$ as the $G$-Banach algebra $\Cont_0(X,  A)$ with the pointwise $G$-action; if $x\in X$, then $\ev_x^A\colon A X \to A$ denotes the (equivariant) evaluation homomorphism at $x$.

Define, for every $G$-Banach algebra $A$,
\begin{eqnarray*}
\zylinder{A} &:= & A[0,1]\\
\cone A & := & A[0,1[\\
\Sigma A & := & A]0,1[ 
\end{eqnarray*}

In the case $A = \C$ we just write $\zylinder$,  $\cone$ and $\suspension$ for the Banach algebras $\zylinder \C$, $\cone \C$ and $\suspension \C$, respectively; they carry the trivial $G$-action.

Two parallel $G$-equivariant morphisms $\varphi_0, \varphi_1 \colon A \to B$ between $G$-Banach algebras are said to be \demph{$G$-equivariantly homotopic} if there exists a $G$-equivariant homotopy from $\varphi_0$ to $\varphi_1$, i.e., a $G$-equivariant morphism $\varphi\colon A \to B[0,1]$ such that $\ev_0^B \circ \varphi = \varphi_0$ and $\ev_1^B \circ \varphi = \varphi_1$. The set of homotopy classes of morphisms from $A$ to $B$ is denoted by $[A,B]$.

Let $A, B$ be $G$-Banach algebras and $m,n \in \Z$. Define
$$
\SigmaHobanG((A,m), (B,n)):=\colim_{k\to \infty} [\Sigma^{m+k} A, \ \Sigma^{n+k} B].
$$
With this set as morphisms, the class of all pairs $(A,m)$ with $A\in \in \GBanAlg$ and $n\in \Z$ becomes a category $\SigmaHobanGcat$. There is a canonical embedding $\can$ of $\GBanAlg/\!\!\sim$ into $\SigmaHobanGcat$ given by $A \mapsto (A,0)$. Define 
$$
\Sigma\colon \SigmaHobanG\to \SigmaHobanG,\ (A,m) \mapsto (A, m+1).
$$
This is an automorphism of the category $\SigmaHobanGcat$. The notation is justified, because there is a natural isomorphism from the functor $A \mapsto (\Sigma A, 0)$ to the functor $A \mapsto (A, 1)$ implemented by $\id_A$, considered as an isomorphism from $(\Sigma A, 0)$ to $(A, 1)$; so $(A,m) \mapsto (A, m+1)$ extends $\Sigma$ from $\GBanAlg/\!\!\sim$ to $\SigmaHobanGcat$.

Let $\varphi \colon A \to B$ be a continuous homomorphism of Banach algebras. The \demph{cone triangle of $\varphi$} is the following diagram in $\BanAlg$ (or its image in $\SigmaHobanGcat$):
$$
\xymatrix{
\Sigma B \ar[r]^-{\iota(\varphi)}  &\cone_{\varphi} \ar[r]^-{\epsilon(\varphi)}  & A \ar[r]^-{\varphi} &B. 
}
$$

A \demph{distinguished triangle} in $\SigmaHobanGcat$ is a $G$-equivariant diagram
$$
\xymatrix{
\Sigma X \ar[r] &X'' \ar[r] & X' \ar[r] &X. 
}
$$
which is isomorphic in $\SigmaHobanGcat$ to the image under $(-\Sigma)^n$, for some $n\in \Z$, of some cone triangle of some continuous $G$-equivariant homomorphism of $G$-Banach algebras.

\begin{theorem}
The category $\SigmaHobanGcat$ together with the inverse suspension $\Sigma^{-1}$ as translation functor and with the class of distinguished triangles defined above is a triangulated category.
\end{theorem}

Let $\xymatrix{B\ar@{>->}[r] & D \ar@{->>}[r]^{\pi}  & A}$ be an extension of Banach algebras; let $\kappa_{\pi}\colon B \to \cone_{\pi}$ be the canonical comparison morphism given by $\kappa_{\pi}\colon B \to \cone_{\pi}, \ b \mapsto (b, 0)$. 

Define 
$$
\mM^G_{\ssplit}:=\{\kappa_{\pi}:\ \xymatrix{B \ \ar@{>->}[r]& D \ar@{->>}[r]^{\pi}& A} \text{ extension in } \mE^G_{\ssplit}\}.
$$
Similarly, define $\mM^G_{G-\ssplit}$ and $\mM^G$.

Given a $G$-equivariant Morita equivalence ${_A}E_B$, consider the canonical injection
$$
\iota_E \colon B \hookrightarrow \left(\begin{matrix}A & E^>\\ E^< & B\end{matrix}\right).
$$
Define 
$$
\mM^G_{\text{Morita}}:=\{\iota_{E}:\ {_A}E_B \text{ $G$-equivariant Morita equivalence}\}.
$$ 

\begin{definition}
Let $\kkbanG$ denote the triangulated category $\SigmaHobanGcat[\mM^G_{G-\ssplit} \cup \mM^G_{\text{Morita}}]$.
\end{definition}

\begin{theorem}
We could also define $\kkbanG$ to be  $\SigmaHobanGcat[\mM^G_{\ssplit} \cup \mM^G_{\text{Morita}}]$ because:
For every equivariant extension $\xymatrix{B\ar@{>->}[r] & D \ar@{->>}[r]^{\pi}  & A}$ of $G$-Banach algebras that admits a (not necessarily equivariant) continuous linear split we have that (the image of) $\kappa_{\pi}$ is invertible in $\kkbanG(B, \cone_{\pi})$.
\end{theorem}
\begin{proof}
Let $\epsilon\colon \xymatrix{B\ar@{>->}[r]^{\iota} & D \ar@{->>}[r]^{\pi}  & A}$ be such an extension, that is, an extension in $\mE^G_{\ssplit}$. By Theorem \ref{Theorem:Extension:GequivariantSplit} the extension
$$
\Leb^1(G\ltimes G, \epsilon)\colon \xymatrix{  \Leb^1(G\ltimes G, B) \ \ar@{>->}[rr]^-{\Leb^1(G\ltimes G, \iota)}&& \Leb^1(G\ltimes G, D) \ar@{->>}[rr]^-{\Leb^1(G\ltimes G,\pi)}& &\Leb^1(G\ltimes G, A)}
$$
has a $G$-equivariant continuous linear split. By Proposition \ref{Proposition:MoritaEquivalenceLeb1}, there is a $G$-equivariant Morita equivalence between $\Leb^1(G\ltimes G, A)$ and $A$. Let $L_A$ denote the corresponding linking algebra. Similarly, define $L_D$ and $L_B$. These algebras fit canonically in a semi-split $G$-equivariant extension $L_{\epsilon}\colon \xymatrix{L_B \ \ar@{>->}[rr]^-{L_{\iota}}&& L_D \ar@{->>}[rr]^-{L_{\pi}} &&L_A}$. We obtain a commutative diagram
$$
\xymatrix{ 
\Leb^1(G\ltimes G, B) \ \ar@{>->}[rr]^-{\Leb^1(G\ltimes G, \iota)} \ar[d]&& \Leb^1(G\ltimes G, D) \ar@{->>}[rr]^-{\Leb^1(G\ltimes G,\pi)} \ar[d]& &\Leb^1(G\ltimes G, A)\ar[d]\\ 
L_B \ \ar@{>->}[rr]^-{L_{\iota}}&& L_D \ar@{->>}[rr]^-{L_{\pi}} &&L_A\\
B\ar@{>->}[rr]^{\iota} \ar[u] && D \ar@{->>}[rr]^{\pi}  \ar[u]&& A\ar[u]
}
$$
Note that all vertical arrows are isomorphisms in $\kkbanG$. We now focus on the lower part of the diagram: It induces a commutative diagram of the form
$$
\xymatrix{ 
B \ar[d]_{\cong} \ar[r]^{\kappa_{\pi}}& \cone_{\pi} \ar[d]^{\gamma}\\
L_B \ar[r]_{\kappa_{L_{\pi}}}& \cone_{L_{\pi}}
}
$$
The vertical arrow $\gamma$ is given by the universal property of the mapping cone; it is straightforward to see that this map is, up to isomorphism of $G$-Banach algebras, the inclusion of a corner in a linking algebra and hence an isomorphism in $\kkbanG$. Because both vertical arrows are isomorphisms in $\kkbanG$, it follows that $\kappa_{\pi}$ is an isomorphism in $\kkbanG$ if and only if $\kappa_{L_{\pi}}$ is. In a similar way, one can show that $\kappa_{L_{\pi}}$ is an isomorphism if and only if $\kappa_{\Leb^1(G\ltimes G, \pi)}$ is an isomorphism in $\kkbanG$. Because the extension $\Leb^1(G\ltimes G, \epsilon)$ is in $\mE^G_{G-\ssplit}$, it follows that $\kappa_{\Leb^1(G\ltimes G, \pi)}$ and hence also $\kappa_{\pi}$ is an isomorphism in $\kkbanG$.
\end{proof}


\begin{deflemma}\label{Deflemma:ExtensionsAreDistinguishedInkkban}
Let $\epsilon\colon \xymatrix{B \ \ar@{>->}[r]& D \ar@{->>}[r]^{\pi}& A}$ be an extension of $G$-Banach algebras in $\mE^G_{\ssplit}$. Then $\kkbanG(\epsilon) \in \kkbanG(\Sigma A, B)$ is defined as the product of the canonical morphism $\Sigma A \to \cone_{\pi}$ in $\kkbanG(\Sigma A, \cone_{\pi})$ and the inverse of the morphism $\kappa_{\pi}\colon B \to \cone_{\pi}$ in $\kkbanG(B,\cone_{\pi})$.

The extension triangle
$$
\xymatrix{\Sigma A  \ar[rr]^{\kkbanG(\epsilon)} && B \ar[r]& D \ar[r] & A}
$$
is a distinguished triangle in $\kkbanG$. In particular, every element of $\mE^G_{\ssplit}$ gives long exact sequences in $\kkbanG$ in both variables.
\end{deflemma}

\subsection{Properties}

The following two lemmas can directly be read of the triangulated structure of $\kkbanG$.

\begin{lemma}\label{Lemma:ExactSequencesWithKContractiveMiddleTerm:kkbanG}
Let $\epsilon\colon \xymatrix{B \ \ar@{>->}[r]& D \ar@{->>}[r]^{\pi}& A}$ be an extension of $G$-Banach algebras in $\mE^G_{\ssplit}$ such that $D \cong 0$ in $\kkbanG$. Then $\kkbanG(\epsilon)$ is an isomorphism from $\Sigma A$ to $B$. 
\end{lemma}

\begin{lemma}\label{Lemma:kkbanVonErweiterungenLeiterdiagramm:kkbanG}
Let $\epsilon\colon \xymatrix{B \ \ar@{>->}[r]& D \ar@{->>}[r]^{\pi}& A}$ and $\epsilon'\colon \xymatrix{B' \ \ar@{>->}[r]& D' \ar@{->>}[r]^{\pi'}& A'}$ be extensions in $\mE^G_{\ssplit}$ which can be put in a diagram of the form
$$
\xymatrix{B \ \ar@{>->}[r] \ar[d]_{\psi}& D \ar@{->>}[r]\ar[d]& A\ar[d]^{\varphi} \\ B' \ \ar@{>->}[r]& D' \ar@{->>}[r]& A'}
$$
Then 
$$
\kkbanG(\Sigma \varphi) \cdot \kkbanG(\epsilon') =  \kkbanG(\epsilon) \cdot \kkbanG(\psi) \ \in\ \kkbanG(\Sigma A, B').
$$
\end{lemma}

\noindent The following theorem is proved just as Theorem~6.14 of \cite{Paravicini:14:kkban:arxiv},  {\it cf.} Theorem~7.26 of \cite{CMR:07}. 

\begin{theorem}[Universal property of $\kkbanG$]\label{Theorem:UniversalProperty:Equivariant} 

Let $F$ be any functor from the category $\GBanAlg$ to an additive category that is $G$-homotopy invariant, $G$-Morita invariant, and half-exact for (equivariantly) semi-split extensions. Then $F$ factors uniquely through $\kkbanGfunc$.

Let $F$ and $F'$ be functors with the above properties, so that they descend to functors $\overline{F}$ and $\overline{F}'$ on $\kkbanG$. If $\Phi\colon F \to F'$ is a natural transformation, then $\Phi$ remains natural with respect to morphisms in $\kkbanG$, that is $\Phi$ is a natural transformation $\overline{F} \to \overline{F}'$.
\end{theorem}

\subsection{The comparison map $\KKbanG\to \kkbanG$}\label{Subsection:Comparison}

\begin{theorem}
There is a canonical bi-natural transformation $\kkbanG$ (in the category of abelian groups) from the bi-functor $\KKbanG(\cdot, \cdot\cdot)$ to the bi-functor $\kkbanG(\cdot, \cdot\cdot)$ on $\GBanAlg$ such that $\kkbanG([\id_B]) = \id_{\kkbanG(B)}$ for all $G$-Banach algebras $B$.  
\end{theorem}

\begin{remark}
This theorem, together with the universal property of $\kkbanG$, implies that, for every functor $\mF \colon \GBanAlg \to \mC$, where $\mC$ is some additive category, that is $G$-homotopy invariant, $G$-Morita invariant and half-exact for semi-split extensions of $G$-Banach algebras, we have a bi-natural transformation $\mF\colon \KKbanG(\cdot, \cdot\cdot) \to \mC(\mF(\cdot), \mF(\cdot\cdot))$ such that $\mF([\id_B]) = \id_{\mF(B)}$ for all $G$-Banach algebras $B$. 

In fact, the proof that we give below shows that the same is true for functors that are  $G$-homotopy invariant, $G$-Morita invariant and just split exact. So if one chooses to define a variant of $\kkbanG$ that is universal for such functors one also has a bi-natural transformation from $\KKbanG$ into this alternative theory.
\end{remark}

\begin{remark}
Although the definition of $\KKbanG$ can be adapted to possibly degenerate Banach algebras, we \emph{confine ourselves to treating only non-degenerate Banach algebras} in what follows to make the presentation clearer; recall that a Banach algebra $B$ is called non-degenerate if the linear span of $B\cdot B$ is dense in $B$. The adaption of the definitions for possibly degenerate Banach algebra given in \cite{Paravicini:14:kkban:arxiv} is left to the interested reader.
\end{remark}

As in the non-equivariant case, we have to consider double split extensions rather than plain extensions (this would be the somewhat simpler case of odd $\KK$-theory); we thus consider quasi-homomorphisms. The double split extensions and quasi-homomorphisms that arise naturally are not $G$-equivariant in a very strict sense (in the first case, one of the splits, and in the second case, one of the homomorphisms is not necessarily $G$-equivariant). This problem seems to be hard to overcome with the trick that worked for plain extensions in Theorem \ref{Theorem:Extension:GequivariantSplit}. But we can go back to the individual $\KKban$-cycle and make it equivariant as in Paragraph \ref{Subsection:EquivariantCycles}.

Let $A$ and $B$ be non-degenerate $G$-Banach algebras and let $(E,T)$ be an element of $\EbanG(A,B)$, i.e., an even $\KKbanG$-cycle.

Assume first that $sT = T$ for all $s\in G$, i.e., assume equivariance of $T$. 

If $T^2 = 1$, then write $E$ as $E_0 \oplus E_1$, by degree. Let $\alpha\colon A \to \Lin_B(E_0) \subseteq \Lin_B(E)$ denote the action of $A$ on $E_0$ and $\bar{\alpha}\colon A \to \Lin_B(E_1) \subseteq \Lin_B(E)$ denote the action of $A$ on $E_1$. Since $T^2 =1$ and $T$ is $G$-equivariant, we have another continuous $G$-equivariant homomorphism:
$$
\Ad_T \circ \bar{\alpha}\colon A \to \Lin_B(E_0) \subseteq \Lin_B(E), \quad a\mapsto T \bar{\alpha}(a) T.
$$
The condition $[a, T] \in \Komp_B(E)$ for all $a\in A$ 
 yields $\Ad_T \circ \bar{\alpha}(a) - \alpha(a) \in \Komp_B(E_0) \subseteq \Komp_B(E)$ for all $a\in A$. Hence we get a quasi-homomorphism
$$
(\alpha, \Ad_T \circ \bar{\alpha})\colon A \rightrightarrows \Lin_B(E) \vartriangleright \Komp_B(E).
$$
This is not quite a $G$-equivariant quasi-homomorphism because the action of $G$ on the Banach algebra $\Lin_B(E)$ is not strongly continuous in general. But we can replace this algebra by the subalgebra of all operators $S\in \Lin_B(E)$ such that $s\mapsto sS$ is continuous. We will not comment on this technical point in what follows and use $\Lin_B(E)$ instead.

This procedure defines an element $\overline{\kkbanG}(E,T)$ of $\kk^{\ban}_{G,0}(A, \Komp_B(E))$ by split-exactness of $\kkbanG$; we write $\overline{\kkbanG}(E,T)$ to distinguish it from the following morphism: As $\Komp_B(E)$ is Morita equivalent to an ideal of $B$, we obtain an element 
$$
\kkbanG(E,T)\quad \in \quad \kk^{\ban}_{G,0}(A, B)
$$

If $T^2 \neq 1$, then use the trick of Lemme 1.2.10 of \cite{Lafforgue:02} that already solved the problem in the non-equivariant case, see \cite{Paravicini:14:kkban:arxiv}, Section~3. This construction preserves equivariance of $T$, and recall that it is compatible with the push-forward and with the sum of cycles. In particular, it respects homotopies. In other words: without loss of generality we can assume that $T^2 =1$.

%

\begin{lemma}\label{Lemma:NaturalTrafoNatural}
Let $(E,T) \in \EbanG(A,B)$ with equivariant $T$ and let $\psi\colon B\to B'$ be an equivariant homomorphism. Then
$$
\kkbanG(\psi_*(E,T)) = \kkbanG(\psi) \circ \kkbanG(E,T)  \ \in \kk^{\ban}_{G,0}(A,B').
$$
\end{lemma}

Let $(E_0,T_0)$ and $(E_1, T_1)$ be $G$-equivariantly homotopic elements of $\E^{\ban}_{G,\equi}(A,B)$. Then the lemma implies that $\kkbanG(E_0,T_0) = \kkbanG(E_1, T_1)$. So we obtain a map
$$
\kkbanG\colon \KK^{\ban}_{G,\equi}(A,B) \to \kkbanG(A,B).
$$
One can show as in \cite{Paravicini:14:kkban:arxiv}, Remark~3.2, Theorem~3.5 and Proposition~3.6 that this construction respects the sum of cycles, that it is functorial in both variables, and that it also respects the action of $G$-equivariant Morita morphisms in the second component, compare \cite{Paravicini:07:Morita:richtigerschienen}.

Now we consider the case that $T$ might not be equivariant.

Let $(E,T) \in \EbanG(A,B)$. Define $\Cont_0(G, (E,T))=(\Cont_0(G, E), \Cont_0(G, T))$ as in Definition \ref{Definition:ContNullVomZykel}. Then $\Cont_0(G, (E,T))$ is a cycle in $\E^{\ban}_{G,\equi}(\Leb^1(G\ltimes G, A), B)$, i.e., such that $\Cont_0(G, T)$ is equivariant. 

Note that $T^2=1$ implies $\Cont_0(G, T)^2=1$, and that the construction which lets us assume that $T^2=1$ is compatible with $T\mapsto \Cont_0(G, T)$. In other words, we can still assume without loss of generality that $\Cont_0(G, T)^2=1$. Recall that the map $(E,T) \mapsto \Cont_0(G, (E,T))$ descends to homomotpy classes and that the resulting natural transformation is denoted by $\gamma_{A,B}$.

Consider the composition
$$
\xymatrix{
\KKbanG(A,B) \ar[rr]^-{\gamma_{A,B}}&& \KK^{\ban}_{G,\equi}(\Leb^1(G\ltimes G, A), B) \ar[rr]^-{\kkbanG}&& \kkbanG(\Leb^1(G\ltimes G, A), B)
}
$$
It is a natural homomorphism. Now $\Leb^1(G\ltimes G, A)$ is $G$-equivariantly and naturally Morita equivalent to $A$, so we obtain a natural isomorphism
$$
\kkbanG(\Leb^1(G\ltimes G, A), B) \cong \kkbanG(A,B).
$$
So we obtain a natural homomorphism
$$
\kkbanG\colon \KKbanG(A,B) \to \kkbanG(A,B).
$$
Note that $\gamma_{A,B}$ respects the right-action on $\KKbanG$ and $\KK^{\ban}_{G,\equi}$ by $G$-equivariant Morita morphisms, so the same holds for the natural homomorphism $\kkbanG$ we have constructed from it.

\begin{proposition}\label{Proposition:AlreadyEquivariantOkay}
In case that the operator $T$ is already equivariant, this more complicated construction yields the same result as the direct construction for equivariant $T$, i.e., the following diagram commutes, where the vertical arrow denotes the obvious forgetful morphism:
$$
\xymatrix{
\KK^{\ban}_{G, \equi}(A,B) \ar[rr]^-{\kkbanG} \ar[d] && \kkbanG(A,B) \\
\KK^{\ban}_{G}(A,B) \ar[rru]_-{\kkbanG}
}
$$
\end{proposition}

\noindent Before we show this proposition, we give a construction that will be useful also later on.

Let $(E,T) \in \EbanG(A,B)$. Then 
$$
\left(\Cont_0(G,E) \oplus E, \ \Cont_0(G,T) \oplus T\right) \in \EbanG( L, B)
$$
where $L$ denotes the linking algebra of the Morita equivalence between $\Leb^1(G\ltimes G, A)$ and $A$; the left action of $L$ on $\Cont_0(G,E)$ is given as follows: the corner $\Leb^1(G\ltimes G, A)\subseteq L$ acts on $\Cont_0(G,E)$ as above (and by $0$ on $E$); the corner $A \subseteq L$ acts on $E$ by the given action and not on $\Cont_0(G,E)$. The subspace $\Cont_0(G,A) \subseteq L$ acts on $E^>$ by the following map
$$
\Cont_0(G,A) \times E^> \to \Cont_0(G,E^>), \ (\alpha, e^>) \mapsto \left[s \mapsto \alpha(s) \ se^> \right];
$$
we also have a multiplication 
$$
\Leb^1(G,E^<) \times \Cont_0(G,A) \to E^<,\ (\xi^<, \alpha) \mapsto \int_{G} \xi^<(r) \ r \alpha(r^{-1}) \rmd r.
$$
The actions of $\Leb^1(G, A)$ on $\Cont_0(G,E^>)$ from the left and on $E^<$ from the right are defined similarly.

If $T$ is equivariant then so is $\Cont_0(G,T) \oplus T$. If $T$ satisfies $T^2=1$ then so does $\Cont_0(G,T) \oplus T$. The construction defines a natural homomorphism
$$
\tilde{\gamma}_{A,B}\colon \KKbanG(A,B) \to \KKbanG(L,B).
$$

\begin{lemma}\label{Lemma:GammaUndGammaTilde} Let $\iota_A$ and $\iota_{\Leb^1}$ denote the inclusions of $A$ and of $\Leb^1(G\ltimes G, A)$ into the linking algebra $L$, respectively. Let $x\in \KKbanG(A,B)$. Then
$$
\iota_A^*(\tilde{\gamma}_{A,B}(x)) = x \LazyAnd \iota_{\Leb^1}^*(\tilde{\gamma}_{A,B}(x)) = \gamma_{A,B}(x) \in \KK^{\ban}_{G}(\Leb^1, B).
$$
This remains true if you replace  $ \KK^{\ban}_{G}$ with $ \KK^{\ban}_{G, \equi}$, everywhere.
\end{lemma}
\begin{proof}
Let $x\in \KKbanG(A,B)$ be represented by $(E,T) \in \EbanG(A,B)$. We construct a $G$-equivariant concurrent homomorphism $\Phi$ from $(E,T)$ to $(\tilde{E}, \tilde{T}):=\left(\Cont_0(G,E) \oplus E, \ \Cont_0(G,T) \oplus T\right)$ with coefficient maps $\iota_A$ and $\id_B$ that will give us a homotopy from $(\id_B)_*(E,T)$ to $\iota_A^*(\tilde{E}, \tilde{T})$. The homomorphism just maps $E$ identically to the summand $E$ of $\tilde{E}$, and it is clear by definition that it intertwines $T$ and $\tilde{T}$. Moreover, the criterion from Theorem~3.1 of \cite{Paravicini:07:Morita:richtigerschienen} is trivially satisfied, so $\Phi$ induces a homotopy as desired.

The same argument works for the second identity, and if $T$ is equivariant, then all operators and homotopies are equivariant, as well.
\end{proof}

\begin{proof}[Proof of Proposition \ref{Proposition:AlreadyEquivariantOkay}.] Consider the following diagram
$$
\xymatrix{
\KK^{\ban}_{G, \equi}(A,B) \ar[dd]_{\gamma_{A,B}} \ar@<-0.5ex>[dr]_{\tilde{\gamma}_{A,B}} \ar[rrrr]^{\kkbanG} &&&& \kkbanG(A,B)\\
& \KK^{\ban}_{G, \equi}(L,B) \ar@<-0.5ex>[ul]_{\iota_A^*} \ar[dl]^{\iota_{\Leb^1}^*} \ar[rr]^{\kkbanG}&&  \kkbanG(L,B)\ar[ur]^{ [\iota_A]\circ} \ar[dr]_{[\iota_{\Leb^1}]\circ}&\\
\KK^{\ban}_{G, \equi}(\Leb^1,B) \ar[rrrr]_{\kkbanG} &&&& \kkbanG(\Leb^1,B) \ar[uu]_{[\iota_{A}]\circ[\iota_{\Leb^1}]^{-1}\circ}\\
}
$$
We have to show that the outer square commutes. Because $\kkbanG$ is a natural transformation, it is clear that the upper square and the lower square commute. The right-hand triangle commutes by definition; note that it is composed of isomorphisms by the definition of $\kkbanG$. The left-hand triangle commutes by the last part of Lemma~\ref{Lemma:GammaUndGammaTilde}. Now the whole diagram commutes because $\iota_A^* \circ \tilde{\gamma}_{A,B}$ is the identity on $\KK^{\ban}_{G, \equi}(A,B) $, a fact that is also contained in Lemma~\ref{Lemma:GammaUndGammaTilde}
\end{proof}

\subsection{The descent homomorphism}
 
Let $\mA(G)$ be an unconditional completion of $\Cont_c(G)$. Consider the functor $\mA(G, \cdot)$ from $\GBanAlg$ to  $\kkban$ given by $B \mapsto \mA(G, B)$ and $\varphi \mapsto \mA(G, \varphi)$.

\begin{proposition}\label{Proposition:DescentOnkkbanG}
The functor $\mA(G, \cdot)$ lifts to a functor from $\kk^{\ban}_G$ to $\kkban$. 
These functors are additive and triangulated.
\end{proposition}
\begin{proof}
To show this result, we just check that $\mA(G,\cdot)\colon \GBanAlg \to \kkban$ satisfies the universal property of $\kk^{\ban}_G$, see Theorem~\ref{Theorem:UniversalProperty:Equivariant}. We already know that it is $G$-homotopy invariant and respects $G$-Morita equivalences, see Proposition~\ref{Proposition:Descent:Morita} and its proof.
\begin{itemize}
\item The functor respects (semi-split) extensions: Let 
$$
\xymatrix{ \epsilon\colon B \ \ar@{>->}[r]& E \ar@{->>}[r]& A}
$$
be a (semi-split) extension of $G$-Banach algebras. In the semi-split case, the continuous linear split $\sigma$ descends to a continuous linear split $\mA(G, \sigma)$ of 
$$
\xymatrix{ \mA(G,\epsilon)\colon \mA(G,B) \ \ar@{>->}[r]& \mA(G,E) \ar@{->>}[r]& \mA(G,A)}.
$$
If the original extension is not semi-split, it nevertheless allows for a continuous $1$-homogeneous split by Michael's selection principle. Such splits also descend under unconditional completions. 

\item The functor respects the suspension: We have to show that, for any $G$-Banach algebra $B$, we have $\mA(G, \Sigma B) \cong \Sigma \mA(G,B)$ in $\kkban$. Consider the short exact sequence
$$
\xymatrix{ \Sigma B \ \ar@{>->}[r]& \cone B \ar@{->>}[r]&B}.
$$
It descends to an extension 
$$
\xymatrix{ \mA(G, \Sigma B) \ \ar@{>->}[r]& \mA(G,\cone B) \ar@{->>}[r]&\mA(G,B)}
$$
with contractible middle term. It hence gives, by Lemma \ref{Lemma:ExactSequencesWithKContractiveMiddleTerm:kkbanG}, an isomorphism $\Sigma \mA(G,B) \cong \mA(G, \Sigma B)$ in $\kkban$.

\item Let $\varphi \colon A \to B$ be a continuous homomorphism of $G$-Banach algebras. Then the canonical homomorphism from $\mA(G, \cone_{\varphi})$ to $\cone_{\mA(G,\varphi)}$ is an isomorphism in $\kkban$: It fits into the commutative diagram
$$
\xymatrix{ 
\mA(G, \Sigma B) \ \ar@{>->}[r]\ar[d]^{\cong}& \mA(G,\cone_\varphi) \ar@{->>}[r]\ar[d]&\mA(G,A)\ar@{=}[d]\\
\Sigma \mA(G, B) \ \ar@{>->}[r]& \cone_{\mA(G,\varphi)} \ar@{->>}[r]&\mA(G,A)\\
}
$$
The left-hand vertical arrow is an isomorphism in $\kkban$ and so is hence the middle vertical arrow.

\item This shows: The functor $\mA(G,\cdot)$ sends cone triangles to distinguished triangles.  

\end{itemize}
So we get a triangulated descent functor from $\kk^{\ban}_G$ to $\kkban$.
\end{proof}

\begin{proposition}
Let $(X,X_0)$ be a finite CW-pair. Then we have a natural isomorphism
$$
\mA(G,\ \Cont((X,X_0), B)) \cong \Cont((X,X_0),\ \mA(G, B))
$$
in $\kkban$ for every $G$-Banach algebra $B$.
\end{proposition}
\begin{proof}
One can prove this just as Proposition~2.7 of \cite{Paravicini:14:kkban:arxiv}. 
\end{proof}

\subsection{Compatibility of the comparison map with the descent}\label{Subsection:CompatibilityComparisonDescent}

\begin{theorem}\label{Theorem:ComparisonAndDescentCompatible:KK0}
Let $A$ and $B$ be non-degenerate $G$-Banach algebras.Then the following diagram is commutative
\begin{equation}\label{Equation:Abstiegskompatibilitaet:kk0}
\xymatrix{
\KK^{\ban}_{G}(A,B) \ar[rr]^-{\mA(G, \cdot)} \ar[d]_{\kkbanG} && \KK^{\ban}(\mA(G,A), \mA(G,B))\ar[d]^{\kkban} \\
\kk^{\ban}_{G}(A,B) \ar[rr]^-{\mA(G, \cdot)} && \kk^{\ban}(\mA(G,A), \mA(G,B))
}
\end{equation}
where the horizontal arrows are the respective descent homomorphisms and the vertical arrows denote the comparison functors between $\KKban$ and $\kkban$.
\end{theorem}

\begin{lemma}
Let $A$ and $B$ be  (non-degenerate) $G$-Banach algebras. Then the following diagram is commutative
\begin{equation}\label{Equation:Abstiegskompatibilitaet:kk0:equi}
\xymatrix{
\KK^{\ban}_{G,\equi}(A,B) \ar[rr]^-{\mA(G, \cdot)} \ar[d]_{\kkbanG} && \KK^{\ban}(\mA(G,A), \mA(G,B))\ar[d]^{\kkban} \\
\kk^{\ban}_{G}(A,B) \ar[rr]^-{\mA(G, \cdot)} && \kk^{\ban}(\mA(G,A), \mA(G,B))
}
\end{equation}
\end{lemma}
\begin{proof}
Let $(E,T) \in \E^{\ban}_{G, \equi}(A,B)$ such that $T^2=1$ (we can assume without loss of generality that this is the case because the construction of \cite{Paravicini:14:kkban:arxiv}, Section~3 is compatible with the descent). 

Consider the following diagram:
$$
\xymatrix{
\mA(G,A)  \ar@{}  |-{ \rightrightarrows} [r] \ar@{=}[d]& \mA(G,\Lin(E)) \ar@{}  |{ \vartriangleright} [r] \ar[d]^{\Psi_{\Lin}}& \mA(G, \Komp(E)) \ar[rrr]^{\kkban(\mA(G,E))} \ar[d]^{\Psi_{\Komp}}&& &Å\mA(G,B) \ar@{=}[d]\\
\mA(G,A) \ar@{}  |-{ \rightrightarrows} [r]  & \Lin(\mA(G,E)) \ar@{}  |{ \vartriangleright} [r]& \Komp(\mA(G, E)) \ar[rrr]^{\kkban(\mA(G,E))}& &&\mA(G,B)
}
$$

If we trace $(E,T)$ in Diagram~(\ref{Equation:Abstiegskompatibilitaet:kk0:equi}) first down and then right, then we obtain the element of $\kkban(\mA(G,A), \mA(G,B))$ given by the upper line in this diagram, i.e., it is the composition of a morphism given by the quasi-homomorphism on the left and the morphism given by $\mA(G,E)$ from $\mA(G,\Komp(E))$ to $\mA(G,B)$. If you trace $(E,T)$ first right and then down you arrive at the corresponding morphism given by the lower line. The lines are connected by the morphisms $\Psi_{\Lin}$ and $\Psi_{\Komp}$, and it is easy to see that the resulting diagram commutes in the obvious sense, compare \cite{Paravicini:14:kkban:arxiv}, Lemma~1.19 and Proposition~2.11. It follows from the cited statements that the elements of $\kkban(\mA(G,A), \mA(G,B))$ given by the two lines agree.
\end{proof}

\begin{proof}[Proof of Theorem~\ref{Theorem:ComparisonAndDescentCompatible:KK0}]
Consider the following diagram
\begin{equation}\label{Equation:Abstiegskompatibilitaet:RiesenDiagramm}
\xymatrix{
\KK^{\ban}_{G}(A,B) \ar[rrr]^-{\mA(G, \cdot)} \ar[d]_{\gamma_{A,B}} &&& \KK^{\ban}(\mA(G,A), \mA(G,B))\ar[ddd]^{\kkban} \ar@{--}[dl] \\
\KK^{\ban}_{G,\equi}(\Leb^1,B) \ar[rr]^-{\mA(G, \cdot)} \ar[d]_{\kkbanG} && \KK^{\ban}(\mA(G,\Leb^1), \mA(G,B))\ar[d]^{\kkban} &\\
\kk^{\ban}_{G}(\Leb^1,B) \ar[rr]^-{\mA(G, \cdot)} \ar[d]_{\cong} && \kk^{\ban}(\mA(G,\Leb^1), \mA(G,B))\ar[dr]^{\cong} &\\
\kk^{\ban}_{G}(A,B) \ar[rrr]^-{\mA(G, \cdot)} &&& \kk^{\ban}(\mA(G,A), \mA(G,B))
}
\end{equation}
Here $\Leb^1$ is an abbreviation of $\Leb^1(G\ltimes G, A)$. The isomorphisms in the lower left-hand square are induced by the Morita equivalence between $A$ and $\Leb^1$ (and its descended version) and it is easy to see that this square commutes. The central square commutes by the preceding lemma. What is left to show is that the remaining diagram is commutative, i.e., the upper left-hand square and the right-hand square. The problem is, that we don't know how to construct an arrow that can be put where the dashed line is and which makes the whole diagram commutative. We hence resort to the following diagram; it suffices to show that it is commutative to see that Diagram~(\ref{Equation:Abstiegskompatibilitaet:RiesenDiagramm}) is commutative as well.
$$
\xymatrix{
\KK^{\ban}_{G}(A,B) \ar@<-0.5ex>[d]_{\tilde{\gamma}_{A,B}} \ar[r] \ar@<-2ex>@/_8ex/[dd]_{\gamma_{A,B}}& \KKban(\mA(G,A), \mA(G,B)) \ar[r]& \kkban(\mA(G,A), \mA(G,B))  \ar@<7ex>@/^8ex/[dd]^{\cong} \\
\KK^{\ban}_{G}(L,B) \ar[d] \ar@<-0.5ex>[u] \ar[r]&\KKban(\mA(G,L), \mA(G,B)) \ar[r]\ar[d]\ar[u] &\kkban(\mA(G,L), \mA(G,B))\ar[d]^{\cong} \ar[u]_{\cong} \\
\KK^{\ban}_{G}(\Leb^1,B) \ar[r]& \KKban(\mA(G,\Leb^1), \mA(G,B)) \ar[r]& \kkban(\mA(G,\Leb^1), \mA(G,B))
}
$$
Here $L$ denotes the linking algebra of $A$ and $\Leb^1$. The left-hand triangle is commutative by Lemma~\ref{Lemma:GammaUndGammaTilde}. The six vertical arrows the sources of which lie in the central column are induced by the injections of $A$ and $\Leb^1$ into the linking algebra $L$; it is hence clear that the four squares in the diagram commute. The left-hand triangle commutes by definition of the morphism on the right-hand side. Now Lemma~\ref{Lemma:GammaUndGammaTilde} implies that the outside hexagon of the diagram commutes which means that Diagram~(\ref{Equation:Abstiegskompatibilitaet:RiesenDiagramm}) is commutative. 
\end{proof}

\section{Equivariant Morita morphisms and the dual Green-Julg theorem}



%

\subsection{Inner actions}\label{Subsection:InnerActions}

\begin{definition}
Let $D$ be a $G$-Banach algebra. Then the action of $G$ on $D$ is called \demph{inner} if there is a strictly continuous and globably bounded homomorphism $U\colon G \to \Mult(D)^{\times}$ such that 
$$
g \cdot d = U_g d U_{g}^{-1}
$$ 
for all $g\in G$ and $d\in D$.
\end{definition}

Let $D$ be a $G$-Banach algebra equipped with an inner action given by a homomorphism $U \colon G \to \Mult(D)^{\times}$; to avoid confusion we denote $D$ by $D_{\kappa}$ when equipped with this inner action. Let $D_{\tau}$ be the Banach algebra $D$ equipped with the trivial $G$-action. Let $D_{\lambda}$ be the Banach algebra $D$ regarded as a Banach $D_{\kappa}$-$D_{\tau}$-bimodule and equipped the the $G$-action given by 
$$
g\cdot d:= U_g d, \quad g\in G, d\in D=D_{\lambda};
$$
let $D_{\rho}$ be the Banach algebra $D$ regarded as a Banach $D_{\tau}$-$D_{\kappa}$-bimodule and equipped the the $G$-action given by 
$$
g\cdot d:= dU_{g}^{-1}, \quad g\in G, d\in D=D_{\rho}.
$$
Now the following lemma can be shown by direct verification of the definitions:
\begin{lemma}\ref{Lemma:InnerActionsMoritaEquivalent}
If you take the multiplication of $D$ as inner products then the pair $(D_{\rho}, D_{\lambda})$ is a $G$-equivariant Morita equivalence between $D_{\kappa}$ and $D_{\tau}$. We will abuse notation and call the Morita equivalence $D_{\lambda}$.
\end{lemma}
In other words, every inner action is naturally Morita equivalent to the trivial action.

\begin{remark}
Note that to be an inner action in the above sense is more than to be an action by inner automorphisms: There are C$^*$-algebras $D$ and unitary actions of $\Z^2$ on $D$ such that $D$ with this action is not equivariantly Morita equivalent to $D$ with the trivial action, but all automorphisms  by which $\Z^2$ acts are inner. The point is that the individual representatives of each automorphism cannot necessarily be chosen in a way to combine to a group homomorphism; in the case of $\Z^2$, the commutators of the representing unitaries will only be central elements and not the identity element.
\end{remark}

If we consider the above situation after taking the $\Leb^1$-descent, we can deduce that $\left(\Leb^1(G,D_{\rho}), \Leb^1(G,D_{\lambda})\right)$ is a Morita equivalence between $\Leb^1(G,D_{\kappa})$ and $\Leb^1(G,D_{\tau})$. But actually, these algebras are already isomorphic as Banach algebras:

Define $\Gamma \colon \Leb^1(G,D_{\kappa}) \to \Leb^1(G,D_{\tau})$ by
$$
f \mapsto \left(t \mapsto f(t)U_t\right).
$$
This is a continuous homomorphism of Banach algebras, and even an isomorphism, the inverse of which has a very similar form. So:
\begin{lemma}
The map $\Gamma$ is a continuous isomorphism of Banach algebras between $ \Leb^1(G,D_{\kappa})$ and $\Leb^1(G,D_{\tau})$. 
\end{lemma}

As a consequence, we now have  two isomorphisms in the Morita category from $ \Leb^1(G,D_{\kappa})$ to $\Leb^1(G,D_{\tau})$: The Morita equivalence defined above and the homomorphism $\Gamma$. In fact, they agree:

\begin{lemma}\label{Lemma:TwoWaysTheSame}
The Morita equivalence $\left(\Leb^1(G,D_{\rho}), \Leb^1(G,D_{\lambda})\right)$ and the homomorphism $\Gamma$ give the same element of $\Moritaban\left(\Leb^1(G,D_{\kappa}),\Leb^1(G,D_{\tau})\right)$.
\end{lemma} 
\begin{proof}
Consider the Morita equivalence $\left(\Leb^1(G,D_{\tau}), \Leb^1(G,D_{\tau})\right)$
 between $\Leb^1(G, D_{\tau})$ and itself. The pair of maps $\left(\Gamma, \id_{\Leb^1(G,D)}\right)$ is a concurrent homomorphism of Morita equivalences between $\left(\Leb^1(G,D_{\rho}), \Leb^1(G,D_{\lambda})\right)$ and $\left(\Leb^1(G,D_{\tau}), \Leb^1(G,D_{\tau})\right)$ with coefficient maps $\Gamma$ and $ \id_{\Leb^1(G,D)_{\tau}}$. It follows from Lemma 1.19 of \cite{Paravicini:14:kkban:arxiv} that 
\begin{eqnarray*}
\left[\left(\Leb^1(G,D_{\rho}), \Leb^1(G,D_{\lambda})\right)\right]&= & [ \id_{\Leb^1(G,D)_{\tau}}]\circ \left[\left(\Leb^1(G,D_{\rho}), \Leb^1(G,D_{\lambda})\right)\right]\\ &=& \left[\left(\Leb^1(G,D_{\tau}), \Leb^1(G,D_{\tau})\right)\right] \circ [\Gamma] = [\Gamma]
\end{eqnarray*}
 in $\Moritaban\left(\Leb^1(G,D_{\kappa}),\Leb^1(G,D_{\tau})\right)$. This shows the claim.
\end{proof}

\subsection{The dual Green-Julg theorem}

\begin{theorem}[The dual Green-Julg Theorem for $\kkban$]\label{Theorem:DualGreenJulg}
Let $G$ be a discrete group. Then
$$
\kkbanG(A, B_{\tau}) \cong \kkban(\ell^1(G, A), B),
$$
naturally, where $A$ is a $G$-Banach algebra and $B$ is a Banach algebra.
\end{theorem}

Note that we write $B_{\tau}$ for the Banach algebra $B$ equipped with the trivial $G$-action.

We are going to show this theorem in a series of lemmas. First note that the theorem states that the functor $\ell^1(G, \cdot)$ from $\kkbanGcat$ to $\kkbancat$ is left adjoint to the functor $B \mapsto B_{\tau}$. We hence just have to produce the correspoding unit and co-unit and show the unit-co-unit equations for this adjunction. The unit and the co-unit will actually be given by Morita morphisms so it suffices to check the equations in this context.

\begin{deflemma}
Let $G$ be a locally compact Hausdorff group and let $A$ be a $G$-Banach algebra and let $g\in G$.

We define a multiplier $U_g=(U_g^<,U_g^>) \in \Mult(\Leb^1(G,A))$ of the Banach algebra $\Leb^1(G,A)$ as follows:
$$
U_g^>(\xi)(s):= g \xi( g^{-1} s), \quad \xi \in \Leb^1(G,A), s\in G,
$$
and
$$
U_g^<(\xi)(s):= \Delta(g^{-1})\xi(sg^{-1}), \quad \xi \in \Leb^1(G,A), s\in G,
$$
where $\Delta \colon G \to \R$ denotes the modular function. We have 
$$
U_{gh} = U_g \circ U_h
$$
for all $g,h\in G$, and $U$ is strictly continuous and globally bounded (by 1). 
\end{deflemma}

\begin{definition}
Let $G$ be a locally compact Hausdorff group. Define an inner action on $\Leb^1(G,A)$ as follows:
If $g\in G$ and $f\in \Leb^1(G,A)$, then define
$$
(gf)(s) := g f(g^{-1}sg), \quad s\in G.
$$
When equipped with this action, we denote $\Leb^1(G,A)$ with $\Leb^1(G,A)_{\kappa}$. We write $\Leb^1(G,A)_{\tau}$ for $\Leb^1(G,A)$ equipped with the trivial action.
\end{definition}

Note that $(gf) = U_g f U_{g^{-1}}$ for all $f\in \Leb^1(G,A)_{\kappa}$ and $g\in G$, so the action is inner. We can hence apply Lemma~\label{Lemma:InnerActionsMoritaEquivalent}:

\begin{lemma}
Let $G$ be a locally compact Hausdorff group. Then $\Leb^1(G,A)_{\kappa}$ and $\Leb^1(G,A)_{\tau}$ are $G$-equivariantly Morita-equivalent. The Morita equivalence is given by the pair (in the notation of Subsection~\ref{Subsection:InnerActions}):
$$
\Leb^1(G,A)_{\lambda}:=(\Leb^1(G,A)_{\rho}, \Leb^1(G,A)_{\lambda})
$$
The action of $G$ on $\Leb^1(G,A)_{\lambda}$ is defined as follows: On the right-hand module $\Leb^1(G,A)_{\lambda}= \Leb^1(G,A)$ it is given by $(g \xi^>) := U_g \xi^>$, on the left-hand module $\Leb^1(G,A)_{\rho}= \Leb^1(G,A)$ it is given by $(g \xi^<) := \xi^< U_{g^{-1}}$.
\end{lemma}

The following lemma that can be checked by direct calculation says that $\Leb^1(G,A)$ together with $g \mapsto U_g$ can be thought of as a covariant representation of the $G$-Banach algebra $A$. We restrict our attention to discrete groups to avoid technical nuisances concerning the relation of $\Leb^1(G,A)$ and its multiplier algebra.

\begin{lemma}
Let $G$ be a discrete group and let $A$ be a $G$-Banach algebra. Then the map
$$
\iota_A\colon A\mapsto \ell^1(G,A), \quad a \mapsto \left[s \mapsto \begin{cases} a & \text{ if $s=e_G$,}\\ 0 & \text{ else,} \end{cases}\right]
$$
is a homomorphism from $A$ to $\ell^1(G,A)$; it is $G$-equivariant as a homomorphism from $A$ to $\ell^1(G,A)_{\kappa}$.
\end{lemma}

Now we have everything in place to define the unit and the co-unit of the adjunction.

Let $A$ be a $G$-Banach algebra and $B$  a Banach algebra. Define
$$
\varepsilon_B \colon \ell^1(G, B_{\tau}) \to B, \ f \mapsto \sum_{s\in G} f(s)
$$
as well as 
$$
\eta_A := \ell^1(G,A)_{\lambda} \circ \iota_A \ \in\  \MoritabanG\left(A,\, \ell^1(G,A)_{\tau}\right).
$$

Note that $\eta_A$ is the composition of (the class of) a homomorphism of Banach algebras and (the class of) a Morita equivalence.

\begin{proposition}
The following (unit-co-unit-) equations hold in $\MoritabanG$ and in $\Moritaban$, respectively: For all $G$-Banach algebras $A$ and all Banach algebras $B$, we have
$$
(\varepsilon_B)_{\tau} \circ \eta_{B_{\tau}} \ = \ \id_{B_{\tau}}
$$
and
$$
\varepsilon_{\ell^1(G,A)} \circ \ell^1(G, \eta_A) \ = \ \id_{\ell^1(G,A)}.
$$
\end{proposition}
\begin{proof}
Consider the following diagram
$$
\xymatrix{
B_{\tau} \ar[r]^-{\iota_{B_{\tau}}}  \ar@{=}[dr]& \ell^1(G,B_{\tau})_{\kappa} \ar[rr]^-{\ell^1(G,B_{\tau})_{\lambda}} \ar[d]_{\varepsilon_B}& \ar@{=>}[d]|{(\varepsilon_B,\varepsilon_B)}& \ell^1(G,B_{\tau})_{\tau} \ar[d]^{\varepsilon_B} \ar[r]^-{\varepsilon_B} & B_{\tau}\\ 
&  B_{\tau} \ar[rr]_{(B_{\tau} ,B_{\tau} )}&&B_{\tau} \ar@{=}[ur]
}
$$
The top horizontal line represents the composition $(\epsilon_B)_{\tau} \circ \eta_{B_{\tau}}$ in $\MoritabanG$. It is easy to see that the left-hand triangle commutes. The right-hand triangle commutes trivially. The central square has the trivial Morita equivalence between $B_{\tau}$ and itself at the bottom, i.e., the Morita equivalence given by $B_{\tau}$ itself that represents the unit morphism in $\MoritabanG(B_{\tau}, B_{\tau})$. The summation homomorphism $\varepsilon_B$ induces a concurrent homomorphism of $G$-Morita equivalences from $\ell^1(G,B_{\tau})_{\lambda}$ to $(B_{\tau} ,B_{\tau} )$, so the central square commutes by Lemma \ref{Lemma:ConcurrentEquivalences:Equivariant}. Hence the composition $(\epsilon_B)_{\tau} \circ \eta_{B_{\tau}}$ is equal to the threefold composition of the identity morphism $\id_{B_{\tau}} \in  \MoritabanG(B_{\tau}, B_{\tau})$ so we have shown the first unit-co-unit equality.

To see that the second equation is true, consider the following diagram:
$$
\xymatrix{
\ell^1(G,A) \ar[r]^-{\ell^1(G,\iota_A)}  \ar[dr]_{\varphi}& \ell^1(G,\ell^1(G,A)_{\kappa}) \ar[rrrr]^-{\ell^1(G,\ell^1(G,A)_{\lambda})} \ar[d]_{\Gamma}&& \ar@{=>}[d]|{(\Gamma,\id)}&& \ell^1(G,\ell^1(G,A)_{\tau}) \ar@{=}[d] \ar[r]^-{\varepsilon_{\ell^1(G,A)}} & \ell^1(G,A)\\ 
&  \ell^1(G,\ell^1(G,A)_{\tau})  \ar[rrrr]_{\ell^1(G,\ell^1(G,A)_{\tau})}&&&&\ell^1(G,\ell^1(G,A)_{\tau}) \ar[ur]_{\varepsilon_{\ell^1(G,A)}}
}
$$
Note that the top row represents the morphism $\varepsilon_{\ell^1(G,A)} \circ \ell^1(G, \eta_A)$.

The homomorphism $\Gamma$ is defined as in Lemma~\ref{Lemma:TwoWaysTheSame}; the same lemma shows that the central square commutes. The homomorphism $\varphi$ is defined so that the left-hand triangle commutes. This means that 
$$
\varphi(f)(t)(s) = \begin{cases}f(t) & \text{ if } t=s,\\ 0 & \text{ else,}\end{cases} \quad \forall f\in \ell^1(G,A), t,s\in G.
$$
The right-hand triangle commutes trivially. So the morphism $\varepsilon_{\ell^1(G,A)} \circ \ell^1(G, \eta_A)$ is the same as the composition $\varepsilon_{\ell^1(G,A)} \circ \varphi$. This homomorphism can be calculated explicitly as follows: 
$$
(\varepsilon_{\ell^1(G,A)} \circ \varphi)(f)(s) = \sum_{t\in G} \varphi(f)(t)(s) = f(s) 
$$
for all $f\in \ell^1(G,A)$ and $s\in G$. So $\varepsilon_{\ell^1(G,A)} \circ \varphi= \id_{\ell^1(G,A)}$.
\end{proof}

\begin{proof}[Proof of Theorem~\ref{Theorem:DualGreenJulg}]
Given that the unit-co-unit equations hold for $\eta$ and $\varepsilon$, the only thing that remains to be checked is that $\eta$ and $\varepsilon$ are natural transformations, where naturality is with respect to $\kk$-elements rather than just homomorphisms of Banach algebras.

Note that, however, it suffices to check naturality only for homomorphisms of Banach algebras because every morphism in $\kkban$ and $\kkbanG$ can be written as the composition of a homomorphism and the inverse of a homomorphism; at least this is true up to suspension, but note that the following diagram commutes for every $G$-Banach algebra $A$:
$$
\xymatrix{
\Sigma A \ar[rr]^{\iota_{\Sigma A}} \ar[rrd]_{\Sigma \iota_A}& & \ell^1(G,\Sigma A)_{\kappa} \ar[d] \ar[rr]^{\ell^1(G,\Sigma A)_{\lambda}}&& \ell^1(G,\Sigma A)_{\tau} \ar[d] \\
&& \Sigma \ell^1(G,A)_{\kappa} \ar[rr]^{\Sigma \ell^1(G,A)} && \Sigma \ell^1(G,A)_{\tau} 
}
$$
Here, the vertical homomorphisms are the canonical maps that are isomorphisms in $\kkban$. So $\eta_{\Sigma A}$ can be identified with $\Sigma \eta_A$. The same is true for $\varepsilon_{\Sigma B}$ and $\Sigma \varepsilon_B$.

The naturality of $\varepsilon$ for homomorphisms is straighforward. To see that $\eta$ is natural, let $\varphi\colon A \to A'$ be a $G$-equivariant homomorphism of Banach algebras. Consider the following diagram
$$
\xymatrix{
\eta_A\colon &A \ar[r]^-{\iota_A} \ar[d]_{\varphi} & \ell^1(G,A)_{\kappa} \ar[rr]^-{\ell^1(G,A)_{\lambda}} \ar[d]_{\ell^1(G,\varphi)_{\kappa}}& \ar@{=>}[d]|{\ell^1(G,\varphi)_{\lambda}}& \ell^1(G,A)_{\tau} \ar[d]^{\ell^1(G,\varphi)_{\tau}}\\ 
\eta_{A'}\colon &A' \ar[r]_-{\iota_{A'}} & \ell^1(G,A')_{\kappa} \ar[rr]_-{\ell^1(G,A')_{\lambda}} && \ell^1(G,A')_{\tau}\\ 
}
$$
The vertical arrows $\ell^1(G,\varphi)_{?}$ are just $\ell^1(G,\varphi)$ in the case $?=\kappa, \tau$ and the concurrent homomorphism $(\ell^1(G,\varphi), \ell^1(G,\varphi))$ in the case $?=\lambda$. That the left-hand square commutes is obvious. Because $\ell^1(G,\varphi)_{\lambda}$ is an equivariant concurrent homomorphism between $G$-Morita equivalences, it follows from Lemma \ref{Lemma:ConcurrentEquivalences:Equivariant} that also the right-hand square commutes on the level of $\MoritabanG$. So $\eta$ is natural for homomorphisms.
\end{proof}

\section{The Green-Julg theorem}

\begin{theorem}[The Green-Julg theorem for $\kkban$]\label{Theorem:GreenJulg}
Let $G$ be a compact Hausdorff group. Then
$$
\kkbanG(A_{\tau}, B) \cong \kkban(A, \Leb^1(G,B)),
$$
naturally, where $A$ is a Banach algebra (denoted by $A_{\tau}$ when equipped with trivial $G$-action) and $B$ is a $G$-Banach algebra.
\end{theorem}

In other words, the functor $B \mapsto \Leb^1(G,B)$ is right-adjoint to the functor $A \mapsto A_{\tau}$. We give the unit and the co-unit of this adjunction:

Let $A$ be a Banach algebra. Define 
$$
\eta_A\colon A \to \Leb^1(G, A_{\tau}), a \mapsto (t \mapsto a).
$$
Note that the constant function $t\mapsto a$ is in $\Cont(G,A)$ and hence in $\Leb^1(G,A)$ for every $a\in A$. It is easy to see that $\eta_A$ is a homomorphism of Banach algebras of norm $\leq 1$. 

Now let $B$ be a $G$-Banach algebra. We have seen above that $B$ is $G$-equivariantly Morita equivalent to $\Leb^1(G\ltimes G, B)$ via the Morita equivalence $(\Leb^1(G,B), \Cont(G,B))$ (we will denote this Morita equivalence just by $\Cont(G,B)$). Consider the $G$-equivariant homomorphism of Banach algebras 
$$
j_B \colon \Leb^1(G,B)_{\tau} \to \Leb^1(G\ltimes G, B), \ f \mapsto \left[(s,t) \mapsto f(s)\right].
$$
We define $\varepsilon_B\in \MoritabanG(\Leb^1(G,B)_{\tau}, B)$ as the composition
$$
\xymatrix{
\varepsilon_B\colon & \Leb^1(G,B)_{\tau} \ar[r]^-{j_B} & \Leb^1(G\ltimes G, B) \ar[rr]^-{\Cont_0(G,B)}_-{\cong} && B. 
} 
$$

As in the proof of Theorem \ref{Theorem:DualGreenJulg} the unit and the co-unit are compatible with the suspension and natural for homomorphisms of Banach algebras. So they are natural with respect to $\rkkban$-elements as well, and what is left to check is that they satisfy the unit-co-unit equations in $\rkkban$. As above, the equations are already true on the level of Morita morphisms:

\begin{proposition}
The following (unit-co-unit-) equations hold in $\MoritabanG$ and in $\Moritaban$, respectively: For all Banach algebras $A$ and all $G$-Banach algebras $B$, we have
$$
\varepsilon_{A_{\tau}} \circ (\eta_A)_{\tau} \ = \ \id_{A_{\tau}}
$$
and
$$
\Leb^1(G, \varepsilon_B) \circ \eta_{\Leb^1(G,B)} \ = \ \id_{\Leb^1(G,B)}.
$$
\end{proposition}
\begin{proof}
For the first equation, Let $A$ be a Banach algebra and consider the homomorphism $\kappa_{A}:=j_{A_{\tau}}\circ (\eta_{A})_{\tau}$ from $A_{\tau}$ to  $\Leb^1(G\ltimes G,A_{\tau})$. It maps an element $a\in A$ to the function $(s,t) \mapsto a$ in $\Cont(G\ltimes G, A)\subseteq \Leb^1(G\ltimes G, A_{\tau})$. Let $K_A= (K_A^<, K_A^>)$ be the concurrent homomorphism from the Morita equivalence $(A_{\tau}, A_{\tau})$ between $A_{\tau}$ and itself to the Morita equivalence $\Cont(G,A)=(\Leb^1(G,A), \Cont(G,A))$ given by the inclusion as constant functions, i.e., $(K_A^>(a))(t):=a$ and $(K_A^<(a))(t):=a$ for all $a\in A$ and $t\in G$. It has coefficient maps $\kappa_A$ and $\id_A$. This situation is summarised in the diagram
$$
\xymatrix{
\Leb^1(G\ltimes G,A_{\tau}) \ar[rrr]^{ \Cont(G,A_{\tau})} & && A_{\tau} \ar@{=}[d]  \\ 
A_{\tau} \ar[rrr]_{(A_{\tau} ,A_{\tau} )} \ar[u]^-{\kappa_{A}} &\ar@{=>}[u]|{(K_A^<,K_A^>)}&& A_{\tau} 
}
$$
By \cite{Paravicini:14:kkban:arxiv}, Lemma~1.19, this implies that $\varepsilon_{A_{\tau}} \circ (\eta_A)_{\tau}$ is, in $\MoritabanG(A_{\tau}. A_{\tau})$, identical to $\id_{A_{\tau}} \circ [(A_{\tau},A_{\tau})] = \id_{A_{\tau}}$.

Now let $B$ be a $G$-Banach algebra. Consider the following diagram:
$$
\xymatrix{
\Leb^1(G, \Leb^1(G\ltimes G, B)) \ar[rrr]^{ \Leb^1(G,\ \Cont(G,B))} & && \Leb^1(G,B) \ar@{=}[d]  \\ 
\Leb^1(G,B) \ar[rrr]_{(\Leb^1(G,B) ,\ \Cont(G,B) )} \ar[u]^-{\phi_B} &\ar@{=>}[u]|{(\Phi_B^<,\Phi_B^>)}&& \Leb^1(G,B) 
}
$$
Here, the homomorphism $\phi$ is the composition $\phi=\Leb^1(G,j_B) \circ \eta_{\Leb^1(G,B)}$ so the composition of the left-hand vertical and upper horizontal arrow is equal to $\Leb^1(G, \varepsilon_B) \circ \eta_{\Leb^1(G,B)}$. The lower horizontal Morita morphism is in fact given by a Morita equivalence. As a pair, it is given by $(\Leb^1(G,B) ,\ \Cont(G,B))$; all operations between elements of $\Leb^1(G,B)$ and $\Cont(G,B)$ are given by convolution. This Morita equivalence induces the identity morphism from $\Leb^1(G,B)$ to $\Leb^1(G,B)$ in the Morita category because there is a canonical concurrent homomorphism from it to the pair $(\Leb^1(G,B), \Leb^1(G,B))$. So what is left to show is that the above diagram is commutative. We do this by producing a concurrent homomorphism $\Phi_B=(\Phi_B^<, \Phi^>_B)$ from $(\Leb^1(G,B) ,\ \Cont(G,B))$ to what we call $\Leb^1(G,\ \Cont(G,B))$ and by which we mean the pair $\left(\Leb^1(G,\ \Leb^1(G,B)), \Leb^1(G,\ \Cont(G,B))\right)$. It is defined as follows:
$$
\Phi_B^>\colon \Cont(G,B) \to \Leb^1(G,\ \Cont(G,B)), \ \beta^> \mapsto \left(t \mapsto \left[ s \mapsto \beta^>(st)\right]\right)
$$
and
$$
\Phi_B^<\colon \Leb^1(G,B) \to \Leb^1(G,\ \Leb^1(G,B)), \ \beta^< \mapsto \left(t \mapsto \left[ s \mapsto \beta^<(s)\right]\right).
$$
Direct calculations show that $\Phi$ is a concurrent homomorphism with coefficient maps $\phi$ and $\id_{\Leb^1(G,B)}$, so \cite{Paravicini:14:kkban:arxiv}, Lemma~1.19, shows that the above diagram commutes. 
\end{proof}

\section{A $\Cont_0(X)$-linear theory}

\subsection{Definition and basic properties}

We can define a theory
$$
\rRkk^{\ban}_G(X; A,B)
$$
where $G$ is a locally compact Hausdorff group, $X$ is a locally compact Hausdorff $G$-space and $A$ and $B$ are $G$-$\Cont_0(X)$-Banach algebras, i.e., $G$-Banach algebras that are also non-degenerate Banach modules over $\Cont_0(X)$ such that all structures are compatible, compare \cite{Paravicini:07,  Paravicini:10:GreenJulg:erschienen, Paravicini:10:SpectralRadius:erschienen}. In particular, $G$-$\Cont_0(X)$-C$^*$-algebras are algebras of this kind.

The definition of the above theory is completely parallel to the definition of $\kk^{\ban}_G$, you just replace all $G$-Banach algebras by $G$-$\Cont_0(X)$-Banach algebras, all continuous linear maps by $\Cont_0(X)$-linear continuous linear maps etc. In particular, Morita equivalences are given by bimodules that are also non-degenerate Banach modules over $\Cont_0(X)$ and such that all operations are compatible. 

If $G$ is trivial then we write $\rRkkban(X; A, B)$ instead of $\rRkk^{\ban}_G(X; A,B)$.

There is a $\Cont_0(X)$-linear version of the universal property and there is a comparison map from $\RKKbanG$ to $\rRkk^{\ban}_G$ just as for $\KKbanG$ and $\kkbanG$. 

In addition, there is an obvious forgetful functor $\mF_X$ from $\rRkkbanG$ to $\kkbanG$; it is a triangulated functor. In formulas, it is often omitted  to avoid clumsy notation. 

Conversely, there is a canonical triangulated functor in the opposite direction, namely the functor $B \mapsto B \otimes \Cont_0(X)$  and $\varphi \mapsto \varphi \otimes \id_{\Cont_0(X)}$. Both functors can be constructed from the universal properties of $\rRkkbanG$ and $\kkbanG$. In fact, they are adjoint functors if $X$ is compact:

\begin{theorem}\label{Theorem:RkkbanForCompactSpace}
Let $X$ be a compact Hausdorff space. Let $A$ be a $G$-Banach algebra and let $B$ be a $G$-$\Cont(X)$-Banach algebra. Then
$$
\rRkkbanG(X; A \otimes \Cont(X), B) \cong \kkbanG(A,B),
$$
naturally. The isomorphism is given by the forgetful map
$$
\rRkkbanG(X; A \otimes \Cont(X), B) \to \kkbanG(A\otimes \Cont(X),B)
$$
followed by the left multiplication with the homomorphism $a \mapsto a \otimes 1$ from $A$ to $A\otimes \Cont(X)$.
\end{theorem}
\begin{proof}
We give the unit and the counit of this adjunction. The unit $\varepsilon_B$ for a $G$-$\Cont(X)$-Banach algebra $B$ is given by
$$
\varepsilon_B \colon B \otimes \Cont(X) \to B, \ b \otimes \chi \mapsto b \chi.
$$
And the counit $\eta_A$ for a $G$-Banach algebra $A$ is given by 
$$
\eta_A \colon A \to A \otimes \Cont(X), \ a \mapsto a \otimes 1.
$$
We check the counit-unit equations: Let $B$ be a $G$-$\Cont(X)$-Banach algebra. Then
$$
\left[\mF_X(\varepsilon_B) \circ \eta_{\mF_X(B)}\right] (b) =\varepsilon_B (b \otimes 1) = b\cdot 1 = b
$$
for all $b\in B$. So $\mF_X(\varepsilon_B) \circ \eta_{\mF_X(B)}= \id_{\mF_X(B)}$.

Now let $A$ be a $G$-$\Cont(X)$-Banach algebra. Then
$$
\left[\varepsilon_{A \otimes \Cont(X)} \circ (\eta_A \otimes \id_{\Cont(X)})\right](a \otimes \chi) = \varepsilon_{A \otimes \Cont(X)} (\eta_A(a) \otimes \chi ) = \varepsilon_{A \otimes \Cont(X)} (a \otimes 1 \otimes \chi ) = a \otimes (1 \chi) = a \otimes \chi
$$
for all $a\in A$ and $\chi \in \Cont(X)$. So $\varepsilon_{A \otimes \Cont(X)} \circ (\eta_A \otimes \id_{\Cont(X)}) = \id_{A\otimes \Cont(X)}$. Thus the theorem is shown.
\end{proof}

\begin{corollary}\label{Corollary:RkkbanAndCompactSpaces}
Let $X$ be a  compact Hausdorff space and let $B$ be a $\Cont(X)$-Banach algebra. Then there is a natural isomorphism
$$
\rRkkban(X; \Cont(X), B) \cong \KTh_0(B).
$$
\end{corollary}

Note that in \cite{Paravicini:14:kkban:arxiv} there is a corresponding result for $\RKKban$, and in fact, the isomorphisms are compatible via the comparison map. The reason is that, also for non-compact $X$, the forgetful map and also the map that sends $B$ to $B \otimes \Cont_0(X)$ are compatible with the comparison map, i.e., we have commutative diagrams
$$
\xymatrix{
\RKKbanG(\Cont_0(X); A,B) \ar[d] \ar[rr]^-{\mF_X}&& \KKbanG(A,B) \ar[d]\\
\rRkkbanG(\Cont_0(X); A,B) \ar[rr]^-{\mF_X}&& \kkbanG(A,B) \\
}
$$
where $A$ and $B$ are $G$-$\Cont_0(X)$-Banach algebras, and
$$
\xymatrix{
\KKbanG(A,B) \ar[d] \ar[rr]^-{\cdot \otimes \Cont_0(X)}&& \RKKbanG(X;A\otimes \Cont_0(X),B\otimes \Cont_0(X)) \ar[d]\\
\kkbanG( A,B) \ar[rr]^-{\cdot \otimes \Cont_0(X)}&& \rRkkbanG(X;A\otimes \Cont_0(X),B\otimes \Cont_0(X)) \\
}
$$
where $A$ and $B$ are $G$-Banach algebras; all vertical maps are supposed to be the respective comparison maps.

Note that the descent homomorphism from Proposition~\ref{Proposition:DescentOnkkbanG} together with the forgetful functor gives a functorial map
$$
\rRkkbanG(\Cont_0(X); A, B) \to \kkban(\mA(G,A), \mA(G,B))
$$
for every unconditional completion $\mA(G)$ of $\Cont_c(G)$. If $X$ is a proper $G$-space, then we can do sligthly better: the space $X/G$ is itself locally compact Hausdorff in this case and we get a functorial map
$$
\rRkkbanG(\Cont_0(X); A, B) \to \rRkkban(X/G;\mA(G,A), \mA(G,B)).
$$
To see this, you just have to add actions of $\Cont_0(X)$ or $\Cont_0(X/G)$ in the appropriate places in the proof of Proposition~\ref{Proposition:DescentOnkkbanG}. This enhanced descent map is compatible with the comparison map between $\RKKbanG$ and $\rRkkbanG$:
\begin{equation}\label{Equation:Abstiegskompatibilitaet:rRkk0}
\xymatrix{
\RKK^{\ban}_{G}(X;A,B) \ar[rr]^-{\mA(G, \cdot)} \ar[d]_{\rRkkbanG} && \RKK^{\ban}(X/G;\mA(G,A), \mA(G,B))\ar[d]^{\rRkkban} \\
\rRkk^{\ban}_{G}(X;A,B) \ar[rr]^-{\mA(G, \cdot)} && \rRkk^{\ban}(X/G;\mA(G,A), \mA(G,B))
}
\end{equation}You prove this fact, again, just by adding actions of $\Cont_0(X)$ and $\Cont_0(X/G)$ in the proof of Theorem~\ref{Theorem:ComparisonAndDescentCompatible:KK0}.

\subsection{A generalised Green-Julg theorem}

Let $G$ be a locally compact Hausdorff group acting continuously and \emph{properly} on a locally compact Hausdorff space $X$. Assume, moreover, that $X/G$ is $\sigma$-compact. This ensures that the $G$-space $X$ posseses a cut-off function, \cite{Tu:04}:

\begin{definition}\label{Definition:CutOffFunction}
A continuous function $c\colon X\to [0,\infty[$ is called \demph{cut-off function} for the $G$-space $X$ if
\begin{enumerate}
\item $\forall x\in X:\ \int_{G} c(t^{-1}x) \rmd t =1$;
\item the restriction of $c$ to any $G$-compact subset of $X$ has compact support.
\end{enumerate}
\end{definition}

Note that, given a cut-off function $c$ on $X$, there is a continuous function $d \colon X \to [0,\infty[$ such that $\norm{d}_{\infty} = 1$, the restriction of $d$ to any $G$-compact subset of $X$ has compact support and $d\equiv 1$ on the support of $c$, compare Proposition~3.2 of \cite{Paravicini:10:GreenJulg:erschienen}. We are going to use both, the function $c$ and the function $d$ later on.

\begin{theorem}[The Green-Julg Theorem for proper group actions]\label{Theorem:GreenJulg:generalised}
Let $G$ be a locally compact Hausdorff group acting properly and continuously on a locally compact Hausdorff space $X$. Then
$$
\rRkkbanG(X; A_{\tau}, B) \cong \rRkkban(X/G; A, \Leb^1(G,B)),
$$
naturally, where $A$ is a $\Cont_0(X/G)$-Banach algebra, $B$ is a $G$-$\Cont_0(X/G)$-Banach algebra and $A_{\tau}:=\Cont_0(X) \otimes_{\Cont_0(X/G)} A$, equipped with the $G$-action coming from the canonical action of $G$ on $\Cont_0(X)$.
\end{theorem}

As above, the theorem says that the functor $B \mapsto \Leb^1(G,B)$ is right-adjoint to the functor $A \mapsto A_{\tau}$. Note that, if $X/G$ is compact and $A$ is trivial, the right-hand side in the above theorem reduces to $\KTh$-theory by Corollary~\ref{Corollary:RkkbanAndCompactSpaces}.

\begin{corollary}\label{Corollary:GGJ} If $X/G$ is compact and $A=\Cont_0(X/G)$ then the above theorem gives a natural isomorphism
$$
\rRkkbanG(X; \Cont_0(X), B)\ \cong \ \rRkkban(X/G; \Cont_0(X/G), \Leb^1(G,B))\  \cong\ \KTh_0(\Leb^1(G,B)).
$$
\end{corollary}

We are now going to construct the unit and the co-unit of the adjunction we want to establish. As above, the unit and the co-unit will be compatible with the suspension and natural for homomorphisms, so they will be natural for $\rRkkban$-elements, too. It will hence suffice to check the unit-co-unit equations.

\subsubsection*{The unit of the adjunction}

Let $A$ be a $\Cont_0(X/G)$-Banach algebra. Essentially. the unit will be a homomorphism 
$$
\eta_A \colon A \to \Leb^1(G, A_{\tau}).
$$
If we regard $A$ as an algebra of sections of a field of Banach algebras over $X/G$, and $A_{\tau}$ as an algebra of sections of the corresponding pull-back field over $X$, then the formula for $\eta_A$ is
$$
[\eta_A (\chi)] (s)(x) = d(x) c(s^{-1}x) \chi([x]),\quad \chi\in \ContSect_c(X/G,A), s\in G, x\in X;
$$
here, the functions $c$ and $d$ are the ones introduced above. But there are two technical obstacles to this approach: Firstly, $A$ might fail to be an algebra of sections if the technical condition of local $\Cont_0(X/G)$-convexity is not met, so a term like $\chi([x])$ might not make sense for $\chi\in A$; secondly, the above formula for $\eta_A$, that is fine for sections $\chi$ with compact support, might not extend to all of $A$. 

We will deal with the first problem by rewriting the above formula without the use of elements $x$ of $X$; the second problem can be resolved by completing ``sections with compact support'' in $A$ to obtain another Banach algebra on which $\eta_A$ is well-defined and that is still sufficiently close to $A$.

To this end, let $A_c$ denote the subspace $\Cont_c(X/G) A$ of $A$. It is a dense subalgebra of $A$ and we think of it as the set of all ``sections with compact support'' in $A$, whereas $A$ itself can be thought of as the ''sections vanishing at infinity''. 

On $A_c$, we define a homomorphism $\phi_A$ to $\Leb^1(G, A_{\tau})$ as follows: Let $a\in A_c$. Find a function $\chi \in \Cont_c(X/G)$ such that $\chi a = a$. Then we define
$$\label{Equation:DefinitionOfphiA}
\phi(a)(s) = d (s\cdot c) \chi \otimes a, \quad s\in G,
$$
where the product of a function $f\in \Cont(X)$ and $\chi \in \Cont(X/G)$ is defined to be $x\mapsto f(x)\chi([x])$. It is easy to see that this definition is independent of the choice of the function $\chi$, so we get a well-defined $\Cont_0(X/G)$-linear homomorphism. We are sometimes going to write $d (s\cdot c) \otimes a$ instead of $d (s\cdot c) \chi \otimes a$ (and similar expressions) in what follows to streamline our notation. Note, however, that $d (s\cdot c)$ is not contained in $\Cont_0(X)$, in general, whereas $d (s \cdot c) \chi$ is. 

Note, moreover, that $\phi_A(a) \in \Cont_c(G, A_{\tau})$ for all $a\in A_c$, but the formula needs not make sense for all $a\in A$ because there is an issue with the norms involved:

Let $a\in A_c$ and $\chi$ as above with $\norm{\chi}_{\infty}= 1$. Then
$$
\norm{\phi_A(a)} = \int_{G} \norm{\phi_A(a)(s)}_{A_{\tau}} \rmd s = \int_{G} \norm{d\  sc\ \chi \otimes a}_{A_{\tau}} \rmd s \leq \int_{G} \norm{d\  sc\ \chi}_{\infty} \rmd s \norm{a}_A.
$$
If we could only interchange the order of integration and supremum, we would arrive at the expression
$$
\sup_{x\in X} \int_{G} \abs{d(x) c(s^{-1}x) \chi([x])} \rmd s \leq \norm{d}_{\infty} \norm{\chi}_{\infty} \int_{G} c(s^{-1}x)\rmd s = 1\cdot 1 \cdot 1 = 1.
$$
But with the given order of integration and supremum, we have to deal with the fact that $\phi_A$ is not bounded in norm as a map from $A_c \subseteq A$ to $\Leb^1(G, A_{\tau})$. 

But we can make it bounded (actually, isometric) by transplanting the norm of $\Leb^1(G, A_{\tau})$ to $A_c$. Define
$$
\norm{a}_0 := \norm{\phi_A(a)}_{\Leb^1(G, A_{\tau})} 
$$
for all $a\in A_c$. Let $A_0$ denote the completion of $A_c$ for this norm. 

\begin{lemma}\label{Lemma:DefinitionOfpsiA}
For all $a\in A_c$ and $\chi'\in \Cont_0(X/G)$, we have $\norm{a}_A \leq \norm{a}_0$ and $\norm{\chi' a}_0 \leq \norm{\chi'}_{\infty} \norm{a}_0$. In particular, the identity on $A_c$ extends to a norm-decreasing $\Cont_0(X/G)$-linear homomorphism 
$$
\psi_A \colon A_0 \to A.
$$
By construction, $\phi_A$ extends to an isometric $\Cont_0(X/G)$-linear homomorphism
$$
\phi_A \colon A_0 \to \Leb^1(G,A_{\tau}).
$$
\end{lemma}
\begin{proof}
First, let $\chi'\in \Cont_0(X/G)$ and $a\in A_c$. Then 
$$
\norm{\chi' a}_0 = \norm{\phi_A(\chi' a)}_{1} = \norm{\chi' \phi_A(a)}_1 \leq \norm{\chi'}_{\infty} \norm{\phi_A(a)}_1 = \norm{\chi'}_{\infty} \norm{a}_0.
$$
Now let $\Leb^1(G\ltimes X)$ denote the completion of the subalgebra $\Cont_c(G, \Cont_c(X))$ of $\Leb^1(G, \Cont_0(X))$ for the norm
$$
\norm{f} := \sup_{x\in X} \int_{G} \abs{f(s)(x)} \rmd s, \quad f\in \Cont_c(G, \Cont_c(X)). 
$$
Note that $\norm{f} \leq \norm{f}_{\Leb^1(G, \Cont_0(X))}$, so there is a canonical homomorphism from $\Leb^1(G, \Cont_0(X))$ to $\Leb^1(G\ltimes X)$. The above calculations show that 
$$
\norm{a}_A = \norm{s\mapsto d \ sc \ \chi}_{\Leb^1(G\ltimes X} \norm{a}_A \leq \norm{(s\mapsto d \ sc \ \chi) \otimes a}_{\Leb^1(G\ltimes X) \otimes_{\Cont_0(X/G)} A},
$$
where $\chi\in \Cont_c(X/G)$ satisfies $\norm{\chi}_{\infty}=1$ and $\chi a = a$. Now there is a canonical isometric isomorphism 
$$
\Leb^1(G, \Cont_0(X)) \otimes_{\Cont_0(X/G)} A \ \cong \  \Leb^1(G, \Cont_0(X) \otimes_{\Cont_0(X/G)} A) = \Leb^1(G, A_{\tau}).
$$
and a canonical norm-decreasing map from $\Leb^1(G, \Cont_0(X)) \otimes_{\Cont_0(X/G)} A$ to $\Leb^1(G\ltimes X) \otimes_{\Cont_0(X/G)} A$. Putting these pieces of information togehter we arrive at
\begin{eqnarray*}
\norm{a}_A &\leq& \norm{(s\mapsto d \ sc \ \chi) \otimes a}_{\Leb^1(G\ltimes X) \otimes_{\Cont_0(X/G)} A}\\
& \leq& \norm{(s\mapsto d \ sc \ \chi) \otimes a}_{\Leb^1(G, \Cont_0(X)) \otimes_{\Cont_0(X/G)} A} = \norm{\phi_A(a)}_{\Leb^1(G,A_{\tau})} = \norm{a}_0.
\end{eqnarray*}
\end{proof}

\begin{lemma}
For all $a\in A_c$ and $a' \in A$ we have 
$$
\norm{aa'}_0 \leq \norm{a}_c \norm{a'}_A
$$
and
$$
\norm{a'a}_0 \leq  \norm{a'}_A\norm{a}_c.
$$
As a consequence, $A_0$ carries a left and a right action of $A$. Using these module actions and the homomorphism $\psi_A$, we can turn the pair $(A_0, A_0)$ into a $\Cont_0(X/G)$-linear Morita equivalence between $A_0$ and $A$. The Morita morphisms induced by $\psi_A$ and $(A_0, A_0)$ agree, more precisely:
$$
[\psi_A] = [(A_0, A_0)] \ \in \ \Moritaban(X/G; A_0, A).
$$
In particular, $\psi_A$ is a Morita isomorphism. 
\end{lemma}
\begin{proof}
We just show the first norm inequality and the equality of the Morita morphisms. We have
$$
\norm{aa'}_c = \norm{\phi_A(aa')} = \int_{G} \norm{d \ sc \ \otimes aa'} \rmd s \leq \int_{G} \norm{d \ sc \ \otimes a}\norm{a'}_A \rmd s = \norm{a}_c \norm{a}_A.
$$
The equality of the Morita morphisms is settled by the commutativity of the following diagram
$$
\xymatrix{
A_0 \ar[rrrr]^{(A_0 ,A_0 )}  \ar[d]_-{\psi_A} &&\ar@{=>}[d]|{(\psi_A,\psi_A)} && A \ar@{=}[d]\\
A \ar[rrrr]_{(A,A )}&&&& A \\
}
$$
\end{proof}

We define $\eta_A\colon A \to \Leb^1(G, A_{\tau})$  to be the composition 
$$
\eta_A:= \phi_A \circ \psi_A^{-1} \quad \in \ \Moritaban(X/G; A, \Leb^1(G, A_{\tau})).
$$

\subsubsection*{The co-unit of the adjunction}

Now let $B$ be a $G$-$\Cont_0(X)$-Banach algebra. We want to define the co-unit of the adjunction as a morphism from the algebra $\Leb^1(G, \Cont_0(X))_{\tau}=  \Cont_0(X)\otimes_{\Cont_0(X/G)}\Leb^1(G,B)$ to the algebra $B$. We first study the co-unit in the case that $B=\Cont_0(X)$: 

Write $\gG$ for the (proper) transport groupoid $G\ltimes X$ (and write $r$ and $s$ for the range and source map of $\gG$, i.e., in our convention, $r(g,x) = x$ and $s(g,x) = g^{-1}x$). This groupoid is the replacement for the compact group $G$ in the non-compact case. We hence want to define a homomorphism from $\Leb^1(G, \Cont_0(X))_{\tau})$ to a Banach algebra such as $\Leb^1(\gG \ltimes \gG)$ that is Morita equivalence to $\Cont_0(X)$. 

We can think of $\Leb^1(G, \Cont_0(X))$ as a groupoid Banach algebra for $\gG$, so we can think of $\Leb^1(G,B)\otimes_{\Cont_0(X/G)} \Cont_0(X)$ as a groupoid Banach algebra for the groupoid $\gG \times_{X/G} X$. The unit space for the latter groupoid is $X \times_{X/G} X$, and the multiplication is induced from the multiplication on $\gG$ and the trivial multiplication on $X$.
Consider the left action of $\gG$ on itself. The resulting transport groupoid $\gG \ltimes \gG$ has morphism space $\gG \times_{X} \gG$  and unit space $X \times_{X/G} X$. The map 
$$
\nu \colon \gG \times_X \gG \to \gG \times_{X/G} X, \ (\gamma_1, \gamma_2) \mapsto (\gamma_1, s(\gamma_2))
$$
is a continuous homomorphism of groupoids and, because $\gG$ is proper, it is a proper map. So we can form the map
$$
\nu^*\colon \Cont_c(\gG \times_{X/G} X) \to \Cont_c(\gG \times_X \gG), \ f \mapsto v \circ \nu.
$$
This map is going to be an isometric homomorphism of Banach algebras from $\Leb^1(\gG \times_{X/G} X)$ to $\Leb^1(\gG \ltimes \gG)$. And, because the groupoids $\gG \ltimes \gG$ and $X$ are equivalent, the latter Banach algebra is Morita equivalent to $\Cont_0(X)$, see \cite{Paravicini:07:Induction:erschienen}. Taking the composition, we obtain a Morita morphism from $\Leb^1(\gG \times_{X/G} X)$ to $\Cont_0(X)$. 

We now switch to the case of a general $G$-$\Cont_0(X)$-Banach algebra $B$, i.e., we consider the above situation with Banach algebraic coefficients. In order to avoid the machinery of fields of Banach algebras and to stay on more familar ground we introduce the following isomorphism: 

The morphisms space $\gG \times_X \gG$ of the groupoid $\gG \ltimes \gG$ can be identified with the space $G \times G \times X$. Instead of looking at $\Leb^1(\gG \ltimes \gG)$ we will analyse $\Leb^1(G \ltimes G, \Cont_0(X))$, or rather, we will analyse $\Leb^1(G \ltimes G, B)$. We know from Proposition~\ref{Proposition:MoritaEquivalenceLeb1} that this algebra is $G$-equivariantly Morita equivalent to $B$, the Morita equivalence being given by the pair $(\Leb^1(G,B), \Cont_0(G,B))$ equipped with the operations as in Proposition~\ref{Proposition:MoritaEquivalenceLeb1}. We have to define actions of $\Cont_0(X)$ on $\Leb^1(G,B)$ and $\Cont_0(G,B)$ that are compatible with the given $\Cont_0(X)$-action on $B$ and the canonical $\Cont_0(X)$-action on $\Leb^1(G \ltimes G, B)$ that is given by 
$$
(f \alpha)(s,t) := t f \cdot \alpha(s,t), \quad f \in \Cont_0(X), \alpha \in \Leb^1(G\ltimes G, B), s,t \in G.
$$
Here we use the action $(t\cdot f)(x)=f(t^{-1}x)$ for all $t\in G$, $f \in \Cont_0(X)$ and $x\in X$. Direct calculation shows that compatible actions are given by
$$
(f \xi^>)(t) := (tf) \cdot \xi^>(t), \quad (f \xi^<)(s) := f \cdot \xi^<(s),
$$
where $f \in \Cont_0(X)$, $\xi^> \in \Cont_0(G,B)$, $\xi^<\in \Leb^1(G,B)$ and $s,t\in G$. 

So $\Leb^1(G \ltimes G, B)$ is $G$-$\Cont_0(X)$-Morita equivalent to $B$. The next step is to find a homomorphism $j_B$ from $\Leb^1(G,B)_{\tau}$ to $\Leb^1(G\ltimes G, B)$. Direct calculation shows that the ansatz
$$
j_B ( f \otimes \beta) (s)(t) := t f \cdot \beta (s)
$$
gives us a norm-contractive $G$-equivariant $\C$-linear and $\Cont_0(X)$-linear homomorphism. 

We can hence define $\varepsilon_B\in \MoritabanG(\Leb^1(G,B)_{\tau}, B)$ as the composition
$$
\xymatrix{
\varepsilon_B\colon & \Leb^1(G,B)_{\tau} \ar[r]^-{j_B} & \Leb^1(G\ltimes G, B) \ar[rrr]^-{(\Leb^1(G,B), \Cont_0(G,B))}_-{\cong} &&& B. 
} 
$$

\subsubsection*{The unit-co-unit equation I}

Let $A$ be a $\Cont_0(X/G)$-Banach algebra. 

\begin{proposition}
We have 
\begin{equation}\label{Equation:GGJ:UnitCoUnit1}
\varepsilon_{A_{\tau}} \circ (\eta_A)_{\tau} \ = \ \id_{A_{\tau}}
\end{equation}
in $\MoritabanG(X;A_{\tau}, A_{\tau})$.
\end{proposition}
\begin{proof}
Choose $c$ and $d$ as in Definition~\ref{Definition:CutOffFunction} and the discussion thereafter and define $\mA(X/G)$, $A_0$, $\phi_A$ and $\psi_A$ as on page \pageref{Equation:DefinitionOfphiA} and Lemma~\ref{Lemma:DefinitionOfpsiA}. 

Define $\kappa_{A}:=j_{A_{\tau}}\circ (\phi_A )_{\tau}$ which is a homomorphism from $(A_0)_{\tau}$ to  $\Leb^1(G\ltimes G,A_{\tau})$. It maps an element $f \otimes \chi a\in (A_0)_{\tau}$, where $f\in \Cont_c(X)$, $\chi \in \Cont_c(X/G)$ and $a\in A$, to the function 
$$
\kappa_A(f\otimes \chi a)\colon G\times G \to A_{\tau},\ (s,t) \mapsto [(t\cdot f) \ d \ (s\cdot c)] \otimes \chi a
$$ 
in $\Cont_c(G\times G, A_{\tau})\subseteq \Leb^1(G\ltimes G, A_{\tau})$. 

Recall that $\psi_A\colon A_0 \to A$ is invertible in $\Moritaban(X/G; A_0, A)$ and can be represented by the following Morita equivalence: the $\Cont_0(X/G)$-Banach algebra $A_0$ carries canonical left and right actions by $A$, so we can turn the pair $(A_0, A_0)$ into a $\Cont_0(X/G)$-linear Morita equivalence between $A_0$ and $A$. 
Likewise, the homomorphism $(\psi_A)_{\tau}$ gives the same Morita isomorphism as the $G$-$\Cont_0(X)$-Morita equivalence $((A_0)_{\tau}, (A_0)_{\tau})$. 
 
We will now produce a concurrent homomorphism $K_A= (K_A^<, K_A^>)$ from the Morita equivalence $((A_0)_{\tau}, (A_0)_{\tau})$ between $(A_0)_{\tau}$ and $A_{\tau}$ to the Morita equivalence $\Cont(G,A_{\tau})=(\Leb^1(G,A_{\tau}), \Cont(G,A_{\tau}))$ with coefficient maps $\kappa_A$ and $\id_A$ that makes the following diagram commutative
%
$$
\xymatrix{
\Leb^1(G\ltimes G,A_{\tau}) \ar[rrr]^{ \Cont(G,A_{\tau})} & && A_{\tau} \ar@{=}[d]  \\ 
(A_0)_{\tau} \ar[rrr]_{((A_0)_{\tau} ,(A_0)_{\tau} )} \ar[u]^-{\kappa_{A}}  & \ar@{=>}[u]|{(K_A^<,K_A^>)}&& A_{\tau} \\
}
$$
Note that this implies Equation~(\ref{Equation:GGJ:UnitCoUnit1}) because the left-hand side of (\ref{Equation:GGJ:UnitCoUnit1}) can be found in the above diagram as the path from the lower right-hand corner to the left, then up and then right.

Define
$$
(K_A^>(f \otimes \chi a))(t):=d \ (t\cdot f) \otimes \chi a
$$ and 
$$
(K_A^<(f \otimes \chi a))(s):=f \ (s \cdot c) \otimes \chi a
$$
for all $f\in \Cont_0(X)$, $\chi \in \Cont_c(X/G)$, $a\in A$ and $s,t\in G$. Direct calculation shows that this definition gives a well-defined norm-contractive $G$-equivariant and $\Cont_0(X)$-linear concurrent homomorphism $K_A$. 
\end{proof}

\subsubsection*{The unit-co-unit equation II}

Let $B$ be a $G$-$\Cont_0(X)$-Banach algebra. 
\begin{proposition}
We have
\begin{equation}\label{Equation:GGJ:UnitCoUnit2}
\Leb^1(G, \varepsilon_B) \circ \eta_{\Leb^1(G,B)} \ = \ \id_{\Leb^1(G,B)}
\end{equation}
in $\Moritaban(X/G; \Leb^1(G,B), \Leb^1(G,B))$.
\end{proposition}
\begin{proof}
Choose $c$ and $d$ as above Definition~\ref{Definition:CutOffFunction} and the discussion thereafter and use the notation of page \pageref{Equation:DefinitionOfphiA}. Define $\lambda_B$ to be the composition
$$
\xymatrix{
\lambda_B \colon \quad  \Leb^1(G,B)_0 \ar[rr]^-{\phi_{\Leb^1(G,B)}}&& \Leb^1(G, \Leb^1(G,B)_{\tau}) \ar[rr]^-{\Leb^1(G,j_B)} && \Leb^1(G, \Leb^1(G\ltimes G, B))
}
$$

Consider the following diagram
$$
\xymatrix{
\Leb^1(G,\Leb^1(G\ltimes G,B)) \ar[rrr]^-{ \Leb^1(G,\Cont_0(G,B))}  &&& \Leb^1(G,B) \ar@{=}[d]  \\ 
\Leb^1(G,B)_0 \ar[rrr]_-{(?^< ,?^> )} \ar[u]^-{\lambda_B}  & \ar@{=>}[u]|{(??^<,??^>)}&& \Leb^1(G,B) \\
}
$$
Here, $\Leb^1(G,\Cont_0(G,B))$ represents the $\Cont_0(X/G)$-linear Morita equivalence $(\Leb^1(G,\Leb^1(G,B)),\Leb^1(G,\Cont_0(G,B)))$, compare Proposition~\ref{Proposition:MoritaEquivalenceLeb1}. We are going to find a suitable $\Cont_0(X/G)$-linear Morita equivalence $(?^<,?^>)$ between $\Leb^1(G,B)_0$ and $\Leb^1(G,B)$ that represents the same Morita morphism as the homomorphism $\psi_{\Leb^1(G,B)}$. What is left to show is that the square diagram is commutative in $\Moritaban(X/G, \cdot, \cdot)$. We do this by producing a concurrent homomorphism $(??^<, ??^>)$ between the Morita equivalences.

We have seen above that the Morita equivalence $(\Leb^1(G,B), \Leb^1(G,B))$ represents $\psi_{\Leb^1(G,B)}$. But this equivalence is not small enough to map into $\Leb^1(G,\Cont_0(G,B))$. We have to construct something smaller:

On $\Leb^1(G,B)$ and on its subspace $\Cont_c(G, B)$, define a left and a right action of $\Cont_0(X)$: 
$$
(\beta^> f)(t) := \beta(t) \ (t\cdot f)
$$
and 
$$
(f \beta^<)(t) := f \ \beta(t)
$$
for all $f\in \Cont_0(X)$, $\beta^< \in \Cont_c(G,B)$, $\beta^> \in \Cont_c(G,B)$ and $t\in G$. Morally, these $\Cont_0(X)$-structures correspond to the fibrations over $X$ of the groupoid $G\ltimes X$ along the source and the range map. 

We define $(\Cont_c(G,B))_c:= \Cont_c(G,B) \Cont_c(X)$ and ${_c}(\Cont_c(G,B)):= \Cont_c(X) \Cont_c(G,B)$. 

Define a linear map $\Lambda_B^>$ on $(\Cont_c(G,B))_c$ with values in $\Leb^1(G, \Cont_0(G,B))$ as follows:
$$
\Lambda_B^>(\beta^>f)(t)(s) := sd \ (\beta^> f)(st) = (s\cdot d) \ (st\cdot f) \ \beta^>(st) 
$$
for $\beta^>\in \Cont_c(G, B)$, $f\in \Cont_c(X)$, $t,s\in G$. Note that $\Lambda_B^>(\beta^> f) \in \Cont_c(G, \Cont_c(G,B))$. We define a new norm on $(\Cont_c(G,B))_c$ as follows:
$$
\norm{\beta^>}_{\rhd} := \max\left\{\norm{\Lambda^>_B(\beta^>)}, \norm{\beta^>}_{\Leb^1(G,B)}\right\}
$$
Let $H^>_B$ denote the completion of $(\Cont_c(G,B))_c$ for this norm. It is elementary though somewhat tiresome to check that 
$$
\norm{\beta *\beta^>}_{\rhd} \leq \norm{\beta}_{\Leb^1(G,B)} \norm{\beta^>}_{\rhd} \LazyAnd \norm{\beta^> *\beta}_{\rhd} \leq  \norm{\beta^>}_{\rhd}\norm{\beta}_{\Leb^1(G,B)},
$$
for $\beta^>, \beta \in (\Cont_c(G,B))_c$. It follows from the density of $(\Cont_c(G,B))_c$ in $\Leb^1(G,B)$ that $H^>_B$ is a Banach $\Leb^1(G,B)$-$\Leb^1(G,B)$-bimodule; it is hence also a  Banach $\Leb^1(G,B)_0$-$\Leb^1(G,B)$-bimodule and thus qualifies as a possibe ``ket-part'' of a Morita equivalence between $\Leb^1(G,B)_0$ and $\Leb^1(G,B)$. It will replace the $?^>$ in the above diagram. 

Similarly, we define a map $\Lambda^<_B$ from ${_c}(\Cont_c(G,B))$ to $\Cont_c(G,\Cont_c(G,B)) \subseteq \Leb^1(G, \Leb^1(G,B))$ as follows:
$$
\Lambda_B^<(f \beta^>)(t)(s) := tc \ (f\beta^< )(s) = (t\cdot c) \ (s\cdot f) \ \beta^<(s) 
$$
for $\beta^<\in \Cont_c(G, B)$, $f\in \Cont_c(X)$, $t,s\in G$. One can show that 
$$
\norm{\Lambda_B^<(\beta^<)} \geq \norm{\beta^<}_{\Leb^1(G,B)},
$$
for all $\beta^< \in {_c}(\Cont_c(G,B))$ so the construction is slightly easier on the ``bra-side''. Define
$$
\norm{\beta^>}_{\lhd} := \norm{\Lambda^<_B(\beta^<)}
$$
for $\beta^> \in {_c}(\Cont_c(G,B))$. Let $H_B^<$ denote the completion of $ {_c}(\Cont_c(G,B))$ for this norm. Again, one can show that the convolution product extends to a Banach $\Leb^1(G,B)$-$\Leb^1(G,B)$-bimodule structure on $H_B^<$ because 
$$
\norm{\beta *\beta^<}_{\lhd} \leq \norm{\beta}_{\Leb^1(G,B)} \norm{\beta^<}_{\lhd} \LazyAnd \norm{\beta^< *\beta}_{\lhd} \leq  \norm{\beta^<}_{\lhd}\norm{\beta}_{\Leb^1(G,B)},
$$
and it follows that $H_B:=(H_B^<, H_B^>)$ can be turned into a Morita equivalence between $\Leb^1(G,B)_0$ and $\Leb^1(G,B)$ (note that the left-hand inner product on $H_B^> \times H_B^<$ takes its values in $\Leb^1(G,B)_0$, canonically). It comes with a concurrent homomorphism to $\Leb^1(G, \Cont_0(G,B))$ by construction, namely $(\Lambda_B^<, \Lambda^>_B)$. It is straightforward to check that $H_B$ and $\Lambda_B$ are compatible with the action of $\Cont_0(X/G)$.

The only thing that is left to check it that $H_B$ represents the homomorphism $\psi_{\Leb^1(G,B)}$ from $\Leb^1(G,B)_0$ to $\Leb^1(G,B)$. But, by construction, there are canonical norm-decreasing linear maps $\Psi_B^>\colon H_B^>\to \Leb^1(G,B)$ and $\Psi_B^<\colon H_B^<\to \Leb^1(G,B)$ extending the identity on $(\Cont_c(G,B))_c$ and ${_c}(\Cont_c(G,B))$, respectively. They fit as a concurrent homomorphism into the following diagram:
$$
\xymatrix{
\Leb^1(G,B)_0 \ar[rrrr]^{(H_B^<,H_B^>)}  \ar[d]_-{\psi_{\Leb^1(G,B)}} &&\ar@{=>}[d]|{(\Psi_B^<,\Psi_B^>)} && \Leb^1(G,B) \ar@{=}[d]\\
\Leb^1(G,B) \ar[rrrr]_{(\Leb^1(G,B), \Leb^1(G,B) )}&&&& \Leb^1(G,B) \\
}
$$
so the result follows.
\end{proof}

\subsection{Connection to the analogous theorem for $\KKban$}

Recall the following result from \cite{Paravicini:10:GreenJulg:erschienen}:

\begin{corollary}[Corollary of Theorem 2.4 of \cite{Paravicini:10:GreenJulg:erschienen}] 
Let $X$ be a proper locally compact Hausdorff $G$-space such that $X/G$ is $\sigma$-compact. Let $B$ be a locally $\Cont_0(X)$-convex $G$-$\Cont_0(X)$-Banach algebra. Then there is a natural isomorphism
$$
\RKKbanG(\Cont_0(X); \Cont_0(X), B) \ \cong \ \RKKban(\Cont_0(X/G); \Cont_0(X/G), \Leb^1(G,B)). 
$$
\end{corollary}

\begin{corollary}[Corollary of Corollary 2.5 of \cite{Paravicini:10:GreenJulg:erschienen}] 
Let $X$ be a proper locally compact Hausdorff $G$-space such that $X/G$ is compact. Let $B$ be a locally $\Cont_0(X)$-convex $G$-$\Cont_0(X)$-Banach algebra. Then there is a natural isomorphism
$$
\RKKbanG(\Cont_0(X); \Cont_0(X), B) \ \cong \ \KTh_0(\Leb^1(G,B)). 
$$
\end{corollary}

\begin{proposition}
Let $X$ be a proper locally compact Hausdorff $G$-space such that $X/G$ is $\sigma$-compact. Let $B$ be a locally $\Cont_0(X)$-convex $G$-$\Cont_0(X)$-Banach algebra. Then the following diagram commutes
$$
\xymatrix{
\RKKbanG(\Cont_0(X); \Cont_0(X), B) \ar[rr]^-{\cong}\ar[d] && \RKKban(\Cont_0(X/G); \Cont_0(X/G), \Leb^1(G,B))\ar[d]\\
\rRkkbanG(\Cont_0(X); \Cont_0(X), B) \ar[rr]^-{\cong} && \rRkkban(\Cont_0(X/G); \Cont_0(X/G), \Leb^1(G,B))\\
}
$$
In particular, we have by \ref{Corollary:RkkbanAndCompactSpaces} and the remark thereafter, for compact $X/G$, that the following diagram commutes
$$
\xymatrix{
\RKKbanG(\Cont_0(X); \Cont_0(X), B) \ar[rr]^-{\cong}\ar[d] && \KTh_0(\Leb^1(G,B)) \\
\rRkkbanG(\Cont_0(X); \Cont_0(X), B) \ar[rru]_-{\cong} && 
}
$$
\end{proposition}
\begin{proof}
The isomorphism from $\RKKbanG(\Cont_0(X); \Cont_0(X), B)$ to $\RKKban(\Cont_0(X/G); \Cont_0(X/G), \Leb^1(G,B))$ of \cite{Paravicini:10:GreenJulg:erschienen} is given by the following device: If $(E,T)$ is in $\EbanG(\Cont_0(X); \Cont_0(X), B)$ such that $T$ is $G$-equivariant (which we can always assume), then we define an embedding of $E_c^>=\Cont_c(X/G) E^>$ into $\Leb^1(G,E^>)$ as follows: 
$$
\Phi_E^>(e^>) := \left[s \mapsto d \ s\cdot e^> \right], \quad e^> \in E_c^>,
$$
where $d$ is as above. We also define an embedding of $E^<_c$ into $\Leb^1(G,E^<)$ by
$$
\Phi_E^<(e^<) := \left[s \mapsto s\cdot c \ e^< \right], \quad e^< \in E_c^<,
$$
We define $E_0^>$, denoted by $\mD(X,E^>)$ in \cite{Paravicini:10:GreenJulg:erschienen}, to be the closure of $E_c^>$ in the pull-back norm from $\Leb^1(G,E^>)$ along $\Phi_E^>$. The space $E_0^<$ is defined analogoulsy. It can be shown that $E_0=(E_0^<, E_0^>)$ is a $\Cont_0(X/G)$-Banach $\Leb^1(G,B)$-pair in a canonical way. The operator $T$ acts on $E_0$, canonically, by the continuous extension $T_0$ of the restriction of $T$ to $E_c$. In \cite{Paravicini:10:GreenJulg:erschienen} it is shown (or rather in \cite{Paravicini:07}), that $(E_0, T_0)$ is in $\Eban(\Cont_0(X/G); \Cont_0(X/G), \Leb^1(G,B))$ and that the map $(E,T) \mapsto (E_0,T_0)$ is an isomorphism on the level of homotopy classes. 

Now consider the following diagram
$$
\xymatrix{
A \ar[rrrr]^{\mD(X,(E,T))} &&&& \Leb^1(G,B)\\
A_0 \ar[rrrr]|{\psi_A^*(E_0,T_0)} \ar[u]^{\psi_A} \ar[d]_{\phi_A}&&\ar@{=>}[u]|{\id}\ar@{=>}[d]|{\Phi_E}&& \Leb^1(G,B) \ar@{=}[u]\ar@{=}[d]\\
\Leb^1(G,A) \ar[rrrr]_{\Leb^1(G,(E,T))} &&&& \Leb^1(G,B)
}
$$
Here we use the notation introduced above, where $A=\Cont_0(X/G)$; recall that $\psi_A$ induced a Morita equivalence and that $\eta_A=\phi_A \circ \psi_A^{-1}$ is the unit of the above adjunction, interpreted as a morphism in the appropriate Morita category. The upper part of the diagram has more or less just illustrative purposes: The top horizontal arrow is the $\RKKban$-element of \cite{Paravicini:10:GreenJulg:erschienen}. The lower part of the diagram should be read as a statement about $\KKban$-elements:  The concurrent homomorphism $\Phi_E$ satisfies the conditions given in Paragraph 4.5.2 of \cite{Paravicini:10:GreenJulg:erschienen}, or rather its obvious variant with coefficient maps that are not the identity, compare \cite{Paravicini:07:Morita:richtigerschienen}, Section~3, and gives us therefore a homotopy between $\phi_A^*(\Leb^1(G,(E,T)))$ and $\psi_A^*(E_0,T_0)$ in $\RKKban(\Cont_0(X/G); A_0, \Leb^1(G,B))$.  If we can show this identity, then transfering it to $\Rkkban$ leads to the equality
$$
\Rkkban(E_0,T_0) \circ \psi_A =\Rkkban(\Leb^1(G,(E,T))) \circ \phi_A,
$$
or
$$
\Rkkban(E_0,T_0)= \Rkkban(\Leb^1(G,(E,T))) \circ \phi_A \circ \psi_A^{-1}  = \Rkkban(\Leb^1(G,(E,T))) \circ \eta_A. 
$$
Now
$$
\Rkkban(\Leb^1(G,(E,T))) \stackrel{(\ref{Equation:Abstiegskompatibilitaet:rRkk0})}{=} \Leb^1(G, \RkkbanG(E,T))
$$
so
$$
\Rkkban(E_0,T_0) = \Leb^1(G, \RkkbanG(E,T)) \circ \eta_A,
$$
i.e., the first part of the proposition is shown. The second follows from the first.
What is left to show is that $\Phi_E$ satisfies indeed the conditions of Theorem~4.15 of \cite{Paravicini:10:GreenJulg:erschienen}. This can be reduced to the following lemma.
\end{proof}

\begin{lemma}
Let $B$ be a $G$-$\Cont_0(X)$-Banach algebra and $E$ be a $G$-$\Cont_0(X)$-Banach $B$-pair. Let $T\in \Lin_B(E)$ be $G$-equivariant, $\Cont_0(X)$-linear and locally compact. Moreover, let $T$ have $G$-compact support in the sense that there exists a function $\chi\in \Cont_c(X/G)$ such that $\chi T = T$. 

Then $T_0$ (as defined above) is in $\Komp_{\Leb^1(G,B)}(E_0)$ and $\phi_A(\chi) \Leb^1(G,T)$ is in $\Komp_{\Leb^1(G,B)}(\Leb^1(G,E))$ and 
$$
\left(T_0, \ \phi_A(\chi) \Leb^1(G,T)\right) \ \in \ \Komp_{\id_{\Leb^1(G,B)}}(\Phi_E),
$$
compare Definition~2.4 of \cite{Paravicini:07:Morita:richtigerschienen} for the notation.
\end{lemma}
\begin{proof}
The proof is straightforward; the only thing one has to know is that $\phi_A(\chi) * \cdot$, as an element in $\Lin_{\Leb^1(G,B)}(\Leb^1(G,E))$, factors through $E_0$. On the right hand side, the factorisation is given by
$$
\phi_A(\chi)* \cdot  = \Phi_E^> \circ \pi_{\chi}
$$
where $\pi_\chi \colon \Leb^1(G,E^>) \to E_0^>$ is given by
$$
\pi_{\chi}(\xi^>) := \int_{G} \chi \circ \pi  \ s \cdot c \ s \xi^<(s^{-1}) \rmd s,
$$
for $\xi^> \in \Cont_C(G,E^>)$; note that $\norm{\pi_\chi} \subseteq \norm{\chi}_{A_0} <\infty$. Now pick $\chi \in \Cont_c(X/G)$ in such a way that $0 \leq \chi \leq 1$ and use $T = \chi^3 T$. 
\end{proof}

\begin{remark}
Note that, in the proof of the lemma, we cannot say that $E_0$ is a direct summand of $\Leb^1(G,E)$. We just show a ``local'' variant of this fact. The norm of $\pi_{\chi}$ can become worse if the support of $\chi$ becomes larger.
\end{remark}

\section{First Poincar{\'e} duality}

In this section, we prove that, under certain conditions, there is a natural isomorphism
$$
\kkban(A\otimes \ell^1(G, \mA) , B) \ \cong \ \kkban(A, B \otimes \ell^1(G, \Cont_0(X))),
$$
for all Banach algebras $A$ and $B$, where $X$ is a $G$-compact proper $G$-space and $\mA$ is a certain proper $G$-Banach algebra; the result is, on an abstract level, a Banach algebraic analogue of the C$^*$-algebraic Poincar{\'e} duality, compare \cite{EEK:07:erschienen, EEK:08, Emerson:02:erschienen, Connes:94}. But there is a technical problem that one has to overcome before one can actually apply this abstract result to actions on manifolds, see Remark~\ref{Remark:ProblemWithPoincare}.

Let $X$ be a $G$-space. For all $G$-Banach algebras $A$ and $B$ we define
$$
\RkkbanG(X; A,B):= \rRkkbanG(X; A\otimes \Cont_0(X), B \otimes \Cont_0(X)).
$$
Note that $\RkkbanG(X; A,B)$ can be thought of as the set off morphisms from $A$ to $B$ in a category $\RkkbanGcat$. The composition in this category is the composition coming from $\rRkkbanGcat$, the identity morphism on some $G$-Banach algebra $A$ is given by $\id_A \otimes \id_{\Cont_0(X)}$.

There is a canonical functor $\sigma_{\{\pt\}, \Cont_0(X)}$ from $\kkbanGcat$ to $\RkkbanGcat$ that is the identity on objects and maps an equivariant homomorphism $\varphi \colon A \to B$ to $\varphi \otimes \id_{\Cont_0(X)}$; note that this functor respects the suspension, sends Morita equivalences to isomorphisms and is also compatible with extensions with continuous linear (equivariant) split, so it extends uniquely to a functor on $\kkbanGcat$. We will abbreviate this functor by $\mG$ (in this paragraph) or $\sigma_{\Cont_0(X)}$ to avoid clumsy notation.

Note that the functor $\mG$ satisfies a ``linearity condition'': If $A$, $B$ and $C$ are $G$-Banach algebras, and $x \in \kkbanG(A,B)$, then 
\begin{equation}\label{Equation:LinearityConditionOfFirstFunctor}
\mG(\id_C \otimes x) = \id_C \otimes\mG(x) \ \in \ \RkkbanG(X; C \otimes A, C\otimes B).
\end{equation}

In the other direction, we can construct functors as follows: Let $\mA$ be a $G$-$\Cont_0(X)$-Banach algebra. There is a canonical homomorphism 
$$
\sigma_{X, \mA}\colon \rRkkbanG(X; A \otimes \Cont_0(X), B \otimes \Cont_0(X)) \to \rRkkbanG(X; A \otimes \mA, B \otimes \mA)
$$
given, on homomorphisms, by $\varphi \mapsto \varphi \otimes_{\Cont_0(X)} \id_{\mA}$. If we compose this with the forgetful functor $\mF_X \colon \rRkkbanGcat \to \kkbanGcat$ we obtain a functor
$$
\mF:=\mF_{\mA}\colon \RkkbanG(X; A, B) \to \kkbanG(A \otimes \mA, B \otimes \mA),
$$
that maps some object $B$ to the object $(B\otimes \Cont_0(X)) \otimes_{\Cont_0(X)} \mA \ \cong \ B\otimes \mA$; in particular, it maps $\C$ to $\mA$. 

Again, we have a ``linearity condition'': If $A$, $B$ and $C$ are $G$-Banach algebras, and $x \in \RkkbanG(X;A,B)$, then 
\begin{equation}\label{Equation:LinearityConditionOfSecondFunctor}
\mF(\id_C \otimes x) = \id_C \otimes\mF(x) \ \in \ \kkbanG(C \otimes A, C\otimes B).
\end{equation}

In certain cases, these two functors are adjoint:

\begin{proposition}[First Poincar{\'e} duality]\label{Proposition:AbstractPoincareDuality:Projective}
Let $\mA$ be a $G$-$\Cont_0(X)$-Banach algebra, $\theta \in \RkkbanG(X; \C, \mA)$ and $D\in \kkbanG(\mA, \C)$ such that $\mG(D) \circ \theta = 1 \in \RkkbanG(\C, \C)$ and such that $\mF_{\mA}(\theta) \in \kkbanG(\mA, \mA \otimes \mA)$ is invariant under the canonical flip isomorphism $\tau_{\mA}$ on $\mA \otimes \mA$.

Then $\mF=\mF_{\mA}$ and $\mG=\sigma_{\Cont_0(X)}$ are adjoint functors, i.e., for all $G$-Banach algebras $A$ and $B$, there is a natural isomorphism
$$
\delta_1^{A,B}\colon \kkbanG(\mF(A), B) \ \cong \ \RkkbanG(X; A, \mG(B))
$$
or, more explicitly,
$$
\delta_1^{A,B}\colon \kkbanG(A \otimes \mA, B) \ \cong \ \rRkkbanG(X; A \otimes \Cont_0(X), B \otimes \Cont_0(X)).
$$
Additionally, if $\eta_A\in \RkkbanG(X; A \otimes \mA, A)$ and $\varepsilon_B\in \kkbanG(B, B \otimes \mA)$ denote the unit and the co-unit of the adjunction, then, for every $G$-Banach algebra $C$, we have
\begin{equation}\label{Equation:LinearityConditionsOfAdjunction}
\eta_{C\otimes A} = \id_C \otimes \eta_A, \quad \varepsilon_{C \otimes B} = \id_C \otimes \varepsilon_B.
\end{equation}
\end{proposition}
\begin{proof}
We have to define the unit and the co-unit of the adjunction and show that they satisfy the usual equations. 

Note that (\ref{Equation:LinearityConditionsOfAdjunction}) implies that $\eta$ and $\varepsilon$ are determined by $\eta_{\C}$ and $\varepsilon_{\C}$, and by (\ref{Equation:LinearityConditionOfFirstFunctor}) and (\ref{Equation:LinearityConditionOfSecondFunctor}) it is actually sufficient to define $\eta_{\C}$ and $\varepsilon_{\C}$ in a way that insures that the unit-co-unit equations are valid for $\C$.

We define
$$
\varepsilon_{\C}:= D \LazyAnd \eta_{\C}:= \theta.
$$
Note that $\mF(\C) = \mA$ and $\mG(\C)= \C$, so $\mF(\mG(\C)) = \mA = \mG(\mF(\C))$. Now
$$
\mG(\varepsilon_{\C}) \circ \eta_{\mG(\C)} = \mG(D) \circ \theta = 1_{\mG(\C)} = \id_{\Cont_0(X)}\in \RkkbanG(X; \C, \C)
$$
by assumption. Moreover, we have
\begin{eqnarray*}
\varepsilon_{\mF(\C)} \circ \mF(\eta_{\C}) &=& \varepsilon_{\mA} \circ \mF(\theta) = (\id_{\mA} \otimes D) \circ \mF(\theta)\\
&=& (\id_{\mA} \otimes D) \circ \tau_{\mA}\circ \mF(\theta) = (D \otimes \id_{\mA}) \circ \mF(\theta) \\
&=& \mF(\mG(D)) \circ \mF(\theta) = \mF(\mG(D) \circ \theta)\\
& =& \mF(1_{\C}) = 1_{\mF(\C)} = \id_{\mA}\in \kkbanG(\mA, \mA).
\end{eqnarray*}
\end{proof}

\begin{corollary}[Compare \cite{Kasparov:88, EEK:07:erschienen}]
Let $G$ be discrete and let $X$ be a $G$-compact proper $G$-space. Let $\mA$, $\theta$ and $D$ be as in Proposition~\ref{Proposition:AbstractPoincareDuality:Projective}. Then, for all Banach algebras $A$ and $B$, there is a natural isomorphism
$$
\kkban(A\otimes \ell^1(G, \mA) , B) \ \cong \ \kkban(A, B \otimes \ell^1(G, \Cont_0(X))),
$$
i.e., $\ell^1(G,\mA)$ is a Poincar{\'e} dual of $\ell^1(G, \Cont_0(X))$. 
\end{corollary}
\begin{proof}
We have a sequence of natural isomorphisms, analogously to Section~1 of \cite{EEK:07:erschienen}:
\begin{eqnarray*}
\kkban(A\otimes \ell^1(G, \mA) , B) &\cong & \kkban( \ell^1(G, A\otimes \mA) , B)\\
&\stackrel{\text{Th.\ }\ref{Theorem:DualGreenJulg}}{\cong} & \kkbanG(A \otimes \mA, B)\\
&\stackrel{\text{Prop.\ }\ref{Proposition:AbstractPoincareDuality:Projective}}{\cong} & \RkkbanG(X; A, B)\\
&=& \rRkkbanG(X; A \otimes \Cont_0(X), B \otimes \Cont_0(X)) \\
&\stackrel{\text{Th.\ }\ref{Theorem:GreenJulg:generalised}}{\cong} & \rRkkban(X/G; A\otimes \Cont_0(X/G), \ell^1(G, B \otimes \Cont_0(X)))\\
&\stackrel{\text{Th.\ }\ref{Theorem:RkkbanForCompactSpace}}{\cong} & \kkban(A, \ell^1(G, B \otimes \Cont_0(X)))\\
&\cong & \kkban(A, B \otimes \ell^1(G, \Cont_0(X))).
\end{eqnarray*}
Note that one can find the unit and the co-unit of this adjunction by composing the units and co-units of the adjunctions appearing in the above sequence of isomorphisms, compare \cite{EEK:07:erschienen}.
\end{proof}

\begin{remark}\label{Remark:ProblemWithPoincare}
Note that the conditions on the element $\theta$ appearing in Proposition~\ref{Proposition:AbstractPoincareDuality:Projective} are rather restrictive. To be an element of $\RkkbanG(X; \C, \mA)$ means to be an element of $\rRkkbanG(X; \Cont_0(X), \mA \otimes \Cont_0(X))$. Now the tensor product that we use for Banach algebras is the projective tensor product, so $\mA \otimes \Cont_0(X)$ is slightly ``smaller'' than $\Cont_0(X,\mA)$, and if $X$ is a $G$-compact proper $G$-manifold, then the element $\theta$ constructed in \cite{Kasparov:88}, Definition~4.4, is in $\rRkkbanG(X; \Cont_0(X), \Cont_0(X, \mA))$ where $\mA = \Cont_{\tau}(X)$. 

There are three possible solutions to this problem, neither of which I have studied so far: One could try to show that the change in tensor product does not matter, compare \cite{Paravicini:14:kkban:arxiv}, Proposition~2.7. Or, one could try to show a version of the Proposition~\ref{Proposition:AbstractPoincareDuality:Projective} that uses the injective tensor product instead of the projective; note that this means, among other things, that we have to argue why $\ell^1(G, \Cont_0(X, B))$ can be identified with $\ell^1(G, \Cont_0(X)) \otimes B$  in the corollary, a fact that might be easier to show if we allow ourselves to invert dense and spectral homomorphisms in $\kkban$. Thirdly, one could try to find variants of the $\theta$ that is usually used that are compatible even with the projective tensor product.
\end{remark}

\nocite{Lafforgue:04, Thomsen:01, Cuntz:81, Cuntz:84, Cuntz:87}

%
\providecommand{\bysame}{\leavevmode\hbox to3em{\hrulefill}\thinspace}
\providecommand{\MR}{\relax\ifhmode\unskip\space\fi MR }
\providecommand{\MRhref}[2]{%
  \href{http://www.ams.org/mathscinet-getitem?mr=#1}{#2}
}
\providecommand{\href}[2]{#2}

\end{document}